\documentclass[12pt,reqno,a4paper]{amsart}
\usepackage{latexsym,amsmath,amsfonts,amssymb}
\newtheorem{Thm}{Theorem}{\bf}{\it}
\newtheorem{Lem}{Lemma}{\bf}{\it}
\newtheorem{Cor}{Corollary}{\bf}{\it}
\newtheorem{Conj}{Conjecture}{\bf}{\it}
\theoremstyle{definition}
\newtheorem*{Ack}{Acknowledgement}
\theoremstyle{remark}
\newtheorem*{Rem}{Remark}
\def\Z{\mathbb Z}

\def\C{\mathbb C}

\def\T{\mathbb T}
\def\R{\mathbb R}
\def\A{\mathbb A}
\def\P{\mathcal P}
\def\M{\mathfrak M}
\def\m{\mathfrak m}
\def\od{\overline{d}}

\def\W{\mathcal W}

\def\1{{\bf 1}}
\def\pmod #1{\ ({\rm mod}\ #1)}

\def\floor #1{\lfloor{#1}\rfloor}
\def\mes {{\rm mes}}

\begin{document}

\title{Difference sets and polynomials of prime variables}
\author{Hongze Li}
\email{lihz@sjtu.edu.cn}
\author{Hao Pan}
\email{haopan79@yahoo.com.cn}

\address{
Department of Mathematics, Shanghai Jiaotong University, Shanghai
200240, People's Republic of China
}
\subjclass[2000]{Primary 11P32; Secondary 05D99, 11P55}
\thanks{This work was supported by
the National Natural Science Foundation of China (Grant No.
10471090).}
\maketitle
\begin{abstract}
Let $\psi(x)$ be a polynomial with rational coefficients. Suppose
that $\psi$ has the positive leading coefficient and zero constant
term. Let $A$ be a set of positive integers with the positive
upper density. Then there exist $x,y\in A$ and a prime $p$ such
that $x-y=\psi(p-1)$. Furthermore, if $P$ is a set of primes with
the positive relative upper density, then there exist $x,y\in P$
and a prime $p$ such that $x-y=\psi(p-1)$.
\end{abstract}

\section{Introduction}
\setcounter{Lem}{0}\setcounter{Thm}{0}\setcounter{Cor}{0}
\setcounter{equation}{0}

For a set $A$ of positive integers, define
$$
\od(A)=\limsup_{x\to\infty}\frac{|A\cap[1,x]|}{x}.
$$
Furstenberg \cite[Theorem 1.2]{Furstenberg1977} and S\'ark\"ozy
\cite{Sarkozy78a} independently confirmed the following conjecture
of Lov\'asz:
\begin{Thm}
\label{diffsquare} Suppose that $A$ is a set of positive integers
with $\od(A)>0$, then there exist $x,y\in A$ and a positive
integer $z$ such that $x-y=z^2$.
\end{Thm}
In fact, the $z^2$ in Theorem \ref{diffsquare} can be replaced by
an arbitrary integral-valued polynomial $f(z)$ with $f(0)=0$. On
the other hand, S\'ark\"ozy \cite{Sarkozy78b} also solved a
problem of Erd\H os:
\begin{Thm}
\label{diffprime} Suppose that $A$ is a set of positive integers
with $\od(A)>0$, then there exist $x,y\in A$ and a prime $p$ such
that $x-y=p-1$.
\end{Thm}
For the further developments of Theorems \ref{diffsquare} and
\ref{diffprime}, the readers may refer to \cite
{Srinivasan85}, \cite{PintzSteigerSzemeredi88},
\cite{BalogPelikanPintzSzemeredi94}, \cite{Green02}, \cite{Lucier06}, \cite{Lucier}, \cite{RuzsaSanders}. In the
present paper, we shall give a common generalization of Theorems
\ref{diffsquare} and \ref{diffprime}. Define
$$
\Lambda_{b,W}=\{x:\,Wx+b\text{ is prime}\}
$$
for $1\leq b\leq W$ with $(b,W)=1$.
\begin{Thm}
\label{diffprimepoly} Let $\psi(x)$ be a polynomial with integral
coefficients and zero constant term. Suppose that $A\subseteq\Z^+$
satisfies $ \od(A)>0$. Then there exist $x, y\in A$ and
$z\in\Lambda_{1,W}$ such that $x-y=\psi(z)$.
\end{Thm}

\begin{Cor}
\label{diffprimepolycor} Let $\psi(x)$ be a polynomial with
rational coefficients and zero constant term. Suppose that
$A\subseteq\Z^+$ satisfies $ \od(A)>0$. Then there exist $x, y\in
A$ and a prime $p$ such that $x-y=\psi(p-1)$.
\end{Cor}
\begin{proof} Let $W$ be the least common multiple of the
denominators of the coefficients of $\psi$. Then the coefficients
of $\psi^*(x)=\psi(Wx)$ are all integers. Then by Theorem
\ref{diffprimepoly}, there exist $x,y\in A$ and
$z\in\Lambda_{1,W}$ such that
$$
x-y=\psi^*(z)=\psi(p-1)
$$
where $p=Wz+1$.
\end{proof}
Quite recently, about one month after the first version of this paper was open in the arXiv server,
in \cite{BergelsonLesigne} Bergelson and Lesigne proved that the set
$$
\{(\psi_1(p-1),\ldots,\psi_m(p-1)):\,p\text{ prime}\}
$$
is an enhanced van der Corput set $\Z^m$,
where $\psi_1,\ldots,\psi_m$ are polynomials with integral coefficients and zero constant term.
Of course, their result can be extended to the set
$\{(\psi_1(z),\ldots,\psi_m(z)):\,z\in\Lambda_{1,W}\}$ without any special difficulty.
On the other hand, Kamae and Mend\'es France \cite{KamaeMendes78} proved that any van der Corput set is also a set of $1$-recurrence.
Hence Bergelson and Lesigne's result also implies our Theorem \ref{diffprimepoly} and Corollary \ref{diffprimepolycor}. In fact, they showed that the set
$\{\psi(p-1):\,p\text{ prime}\}$ is not only a set of $1$-recurrence, but also a set of strong $1$-recurrence.

For two sets $A, X$ of positive integers, define
$$
\od_{X}(A)=\limsup_{x\to\infty}\frac{|A\cap
X\cap[1,x]|}{|X\cap[1,x]|}.
$$
Let $\P$ denote the set of all primes. In \cite{Green05}, Green
established a Roth's-type extension of a result of van der Corput
\cite{Corput39} on 3-term arithmetic progressions in primes:

\medskip
{\it Let $P$ be a set of primes with $\od_{\P}(P)>0$, then there
exists a non-trivial 3-term arithmetic progressions contained in
$P$.}
\medskip

The key of Green's proof is a transference principle, which
transfers a subset $P\subseteq\P$ to a subset
$A\subseteq\Z_N=\Z/N\Z$ with $|A|/N\geq \od_\P(P)/64$, where $N$ is
a large prime. Using Green's ingredients, now we can show that:
\begin{Thm}
\label{primediffprimepoly} Let $\psi(x)$ be a polynomial with
integral coefficients and zero constant term. Suppose that
$P\subseteq\P$ satisfies $ \od_\P(P)>0$. Then there exist $x, y\in
P$ and $z\in\Lambda_{1,W}$ such that $x-y=\psi(z)$.
\end{Thm}
Similarly, we have
\begin{Cor}
\label{primediffprimepolycor} Let $\psi(x)$ be a polynomial with
rational coefficients and zero constant term. Suppose that
$P\subseteq\P$ satisfies $ \od_\P(P)>0$. Then there exist $x, y\in
P$ and a prime $p$ such that $x-y=\psi(p-1)$.
\end{Cor}
On the other hand, the well-known Szemer\'edi theorem
\cite{Szemeredi75} asserts that for any set $A$ of positive
integers with $\od(A)>0$, there exist arbitrarily long arithmetic
progressions contained in $A$. In \cite{BergelsonLeibman96},
Bergelson and Leibman extended Theorem \ref{diffsquare} and
Szemer\'edi's theorem:

\medskip
{\it Let $\psi_1(x),\ldots,\psi_m(x)$ be arbitrary integral-valued
polynomials with $\psi_1(0)=\cdots=\psi_m(0)=0$. Then for any set
$A$ of positive integers with $\od(A)>0$, there exist $x\in A$ and
a integer $z$ such that $x+\psi_1(z),\ldots,x+\psi_m(z)$ are all
contained in $A$.}
\medskip

Recently, Tao and Ziegler \cite{TaoZiegler} proved that:

\medskip
{\it Let $\psi_1(x),\ldots,\psi_m(x)$ be arbitrary integral-valued
polynomials with $\psi_1(0)=\cdots=\psi_m(0)=0$. Then for any set
$P$ of primes with $\od_\P(P)>0$, there exist $x\in P$ and a
integer $z$ such that $x+\psi_1(z),\ldots,x+\psi_m(z)$ are all
contained in $P$.}
\medskip

This is a generalization of Green and Tao's celebrated result
\cite{GreenTao1} that the primes contain arbitrarily long
arithmetic progressions. Furthermore, with the help of a very deep
result due to Green and Tao \cite{GreenTao2} on the Gowers norms
\cite{Gowers01}, Frantzikinakis, Host and Kra
\cite{FrantzikinakisHostKra} proved that if $\od(A)>0$ then $A$
contains a 3-term arithmetic progression with the difference
$p-1$, where $p$ is a prime. In fact, using the methods of Green
and Tao in \cite{GreenTao2}, it is not difficult to replace $A$ by
$P$ with $\od_\P(P)>0$ in the result of Frantzikinakis, Host and
Kra.

Motivated by the above results, here we propose two conjectures:
\begin{Conj}
Let $\psi_1(x),\ldots,\psi_m(x)$ be arbitrary polynomials with
rational coefficients and zero constant terms. Then for any set
$A$ of positive integers with $\od(A)>0$, there exist $x\in A$ and
a prime $p$ such that $x+\psi_1(p-1),\ldots,x+\psi_m(p-1)$ are all
contained in $A$.
\end{Conj}
\begin{Conj}
Let $\psi_1(x),\ldots,\psi_m(x)$ be arbitrary polynomials with
rational coefficients and zero constant terms. Then for any set
$P$ of primes with $\od_\P(P)>0$, there exist $x\in P$ and a prime
$p$ such that $x+\psi_1(p-1),\ldots,x+\psi_m(p-1)$ are all
contained in $P$.
\end{Conj}
The proofs of Theorems \ref{diffprimepoly} and
\ref{primediffprimepoly} will be given in section 3 and section 4.
Throughout this paper, without the additional mentions, the
constants implied by $\ll$, $\gg$ and $O(\cdot)$ will only depend
on the degree of $\psi$.

\section{Some Necessary Lemmas on Exponential Sums }

Let $\T$ denote the torus $\R/\Z$. For any function $f$ over $\Z$,
define $f^\Delta(x)=f(x+1)-f(x)$. Also, we abbreviate
$e^{2\pi\sqrt{-1}x}$ to $e(x)$. Let $\psi(x)=a_1x^k+\cdots+a_kx$ be
a polynomial with integral coefficients. In this section, we always
assume that $W, |a_1|,\ldots,|a_k|\leq \log N$.
\begin{Lem}
\label{polymajor} Suppose that $h(x)$ is an arbitrary polynomial
and $0<\nu<1$. Then for any $\alpha\in\T$
$$
\sum_{x=1}^Nh(x)e(\alpha\psi(x))=\frac{1}{q}\sum_{r=1}^qe(a\psi(r)/q)\sum_{x=1}^Nh(x)e((\alpha-a/q)\psi(x))+O_{\deg
h}(h(N)N^{\nu})
$$
provided that $|\alpha q-a|\leq N^\nu/\psi(N)$ with $1\leq a\leq
q\leq N^\nu$.
\end{Lem}
\begin{proof}
Let $\theta=\alpha-a/q$. Then by a partial summation, we have
\begin{align*}
&\sum_{x=1}^Nh(x)e(a\psi(x)/q)e(\theta\psi(x))\\
=&h(N)e(\theta\psi(N))F_N(a/q)
-\sum_{y=1}^{N-1}(h(y+1)e(\theta\psi(y+1))-h(y)e(\theta\psi(y)))F_y(a/q),
\end{align*}
where
$$
F_y(a/q):=\sum_{x=1}^ye(a\psi(x)/q)=\frac{y}{q}\sum_{r=1}^qe(a\psi(r)/q)+O(q).
$$
Clearly
\begin{align*}
&h(y+1)e(\theta\psi(y+1))-h(y)e(\theta\psi(y))\\
=&
(h(y+1)-h(y))e(\theta\psi(y+1))+h(y)e(\theta\psi(y))(e(\theta\psi^\Delta(y))-1)\\
=&O(h^\Delta(y))+O(h(y)\theta\psi^\Delta(y)).
\end{align*}
This concludes that
\begin{align*}
&\sum_{x=1}^Nh(x)e(a\psi(x)/q)e(\theta\psi(x))\\
=&\frac{1}{q}\sum_{r=1}^qe(a\psi(r)/q)\sum_{x=1}^Nh(x)e(\theta\psi(x))+O(\theta
qN\psi^\Delta(N)h(N))+O(qh^\Delta(N)N).
\end{align*}
\end{proof}
Define
$$
\lambda_{b,W}(x)=\begin{cases}
\frac{\phi(W)}{W}\log(Wx+b)&\qquad\text{if }Wx+b\text{ is prime},\\
0&\qquad\text{otherwise},
\end{cases}
$$
where $\phi$ is the Euler totient function.
\begin{Lem}
\label{primepolymajor} Suppose that $h(x)$ is an arbitrary
polynomial and $B>1$. Then for any $\alpha\in\T$
\begin{align*}
&\sum_{x=1}^Nh(x)\lambda_{b,W}(x)e(\alpha\psi(x))\\
=&\frac{\phi(W)}{\phi(Wq)}\sum_{\substack{1\leq r\leq q\\
(Wr+b,q)=1}}e(a\psi(r)/q)\sum_{x=1}^Nh(x)e((\alpha-a/q)\psi(x))+O_{\deg
h}(h(N)Ne^{-c\sqrt{\log N}})
\end{align*} provided that $|\alpha q-a|\leq (\log N)^B/\psi(N)$
with $1\leq a\leq q\leq (\log N)^B$, where $c$ is a positive
constant.
\end{Lem}
\begin{proof}
Let
\begin{align*} F_y(a/q)=&\sum_{x=1}^y\lambda_{b,W}(x)e(a
\psi(x)/q)\\=&\sum_{\substack{1\leq r\leq Wq\\ (r,q)=1\\ r\equiv
b\pmod{W}}}e(a\psi((r-b)/W)/q)\sum_{\substack{x\in \Lambda_{r,Wq}\\
Wqx+r\leq Wy+b}}\frac{\phi(W)q}{\phi(Wq)}\lambda_{r,Wq}(x).
\end{align*}
The well-known Siegel-Walfisz theorem (cf.
\cite{Davenport00}) asserts that
$$
\sum_{\substack{p\leq y\text{ is prime}\\ p\equiv b\pmod{q}}}\log
p=\frac{y}{\phi(q)}+O(ye^{-c'\sqrt{\log y} })
$$
provided that $q\leq\log^{c_1} y$, where $c_1,\,c'$ are  positive
constants. Hence
$$
\sum_{\substack{x\in \Lambda_{r,Wq}\\ Wqx+r\leq
Wy+b}}\lambda_{r,Wq}(x)=\frac{y}{q}+O(Wye^{-c'\sqrt{\log (Wy)}}).
$$
It follows that
$$
F_y(a/q)=\frac{\phi(W)y}{\phi(Wq)}\sum_{\substack{1\leq r\leq q\\
(Wr+b,q)=1}}e(a\psi(r)/q)+O(ye^{-c'\sqrt{\log y}/2}).
$$
Let $\theta=\alpha-a/q$. Then
\begin{align*}
&\sum_{x=1}^Nh(x)\lambda_{b,W}(x)e(\alpha\psi(x))\\
=&h(N)e(\theta\psi(N))F_{N}(a/q)
-\sum_{y=1}^{N-1}(h(y+1)e(\theta \psi(y+1))-h(y)e(\theta \psi(y)))F_y(a/q)\\
=& \frac{\phi(W)}{\phi(Wq)}\sum_{\substack{1\leq r\leq q\\
(Wr+b,q)=1}}e(a\psi(r)/q) \sum_{y=1}^{N}h(y)e(\theta
\psi(y))+O(h(N)Ne^{-c'\sqrt{\log N}/3})
\end{align*}
by noting that
$$
h(y+1)e(\theta \psi(y+1))-h(y)e(\theta
\psi(y))=O(h^\Delta(y))+O(h(y)\theta\psi^\Delta(y+1)).
$$
\end{proof}
\begin{Lem}
\label{polydelta} For any $\theta\in\T$,
$$
\sum_{x=1}^N\psi^\Delta(x-1)e(\theta\psi(x))=\sum_{x=1}^{\psi(N)}e(\theta
x)+O(\theta\psi(N)\psi^\Delta(N)).
$$
\end{Lem}
\begin{proof} Clearly
\begin{align*}
\sum_{x=1}^N\psi^\Delta(x-1)e(\theta\psi(x))-\sum_{x=1}^{\psi(N)}e(\theta
x)
=&\sum_{x=1}^Ne(\theta\psi(x))\sum_{y=0}^{\psi^\Delta(x-1)-1}(1-e(-\theta y))\\
=&O\bigg(\sum_{x=1}^N\sum_{y=0}^{\psi^\Delta(x-1)-1}\theta y\bigg)\\
=&O(\theta\psi(N)\psi^\Delta(N)).
\end{align*}
\end{proof}
\begin{Lem} For any $\epsilon>0$,
\label{polyminor}
$$
\sum_{x=1}^Ne(\alpha\psi(x))\ll_\epsilon
N^{1+\epsilon}\bigg(\frac{a_1}{q}+\frac{a_1}{N}+\frac{q}{N^k}\bigg)^{2^{1-k}}
$$
provided that $|\alpha-a/q|\leq q^{-2}$.
\end{Lem}
\begin{proof} We left the proof of Lemma \ref{polyminor} as
an exercise for the readers,
since it is just a little modification of the proof
of Wely's inequality \cite[Lemma 2.4]{Vaughan97}.
\end{proof}

\begin{Lem}[Hua]
\label{huasum} Suppose that $(q,a_1,\ldots,a_k)=1$. Then
$$
\sum_{r=1}^qe(\psi(r)/q)\ll_\epsilon q^{1-\frac{1}{k}+\epsilon}
$$
for any $\epsilon>0$.
\end{Lem}
\begin{proof} See \cite[Theorem 7.1]{Vaughan97}.
\end{proof}

\begin{Lem}
\label{huapoly}
$$
\int_\T\bigg|\sum_{x=1}^N\psi^\Delta(x-1)e(\alpha\psi(x))\bigg|^\rho
d\alpha\ll_\rho\gcd(\psi)\psi(N)^{\rho-1}
$$
for $\rho\geq k2^{k+2}$, where $\gcd(\psi)$ denotes the greatest
common divisor of $a_1,\ldots,a_k$.
\end{Lem}
\begin{proof}
Notice that
\begin{align*}
\int_0^1\bigg|\sum_{x=1}^N(a\psi)^\Delta(x-1)e(\alpha
a\psi(x))\bigg|^\rho d\alpha
=&a^{\rho-1}\int_0^a\bigg|\sum_{x=1}^N\psi^\Delta(x-1)e(\alpha \psi(x))\bigg|^\rho d\alpha\\
=&a^{\rho}\int_0^1\bigg|\sum_{x=1}^N\psi^\Delta(x-1)e(\alpha
\psi(x))\bigg|^\rho d\alpha.
\end{align*}
So without loss of generality, we may assume that $\gcd(\psi)=1$.
Let $\nu=1/5$ and $\epsilon=2^{-k}\nu-\frac{k}{2\rho}$. Let
$$
\M_{a,q}=\{\alpha\in\T:\,|\alpha q-a|\leq N^\nu/\psi(N)\},
\quad\M=\bigcup_{\substack{1\leq a\leq q\leq
N^\nu\\(a,q)=1}}\M_{a,q}
$$
and $\m=\T\setminus\M$. Clearly $\mes(\M)\leq N^{3\nu}/\psi(N)$,
where $\mes(\M)$ denotes the Lebesgue measure of $\M$.

If $\alpha\in\m$, then by Lemma \ref{polyminor} we have
\begin{align*}
&\sum_{x=1}^N\psi^\Delta(x-1)e(\alpha\psi(x))\\
=&\psi^\Delta(N-1)\sum_{x=1}^Ne(\alpha\psi(x))-\sum_{y=1}^{N-1}(\psi^\Delta(y)-\psi^\Delta(y-1))\sum_{x=1}^ye(\alpha\psi(x))\\
\ll&_{\epsilon}\psi^\Delta(N)N^{1+\epsilon-2^{1-k}\nu}.
\end{align*}
Hence
\begin{align*}
\int_\m\bigg|\sum_{x=1}^N\psi^\Delta(x-1)e(\alpha\psi(x))\bigg|^\rho
d\alpha\ll_{\epsilon}\psi(N)^\rho N^{\rho(\epsilon-2^{1-k}\nu)}=o(
\psi(N)^{\rho-1}).
\end{align*}

On the other hand, when $\alpha\in\M$, by Lemmas \ref{polymajor}
and \ref{polydelta},
\begin{align*}
\sum_{x=1}^N\psi^\Delta(x-1)e(\alpha\psi(x))
=\frac{1}{q}\sum_{r=1}^qe(a\psi(r)/q)\sum_{x=1}^{\psi(N)}e((\alpha-a/q)x)+O(\psi^\Delta(N)N^{\nu}).
\end{align*}
Let $L=\floor{\rho/2}$. Obviously
\begin{align*}
\int_\M\bigg|\sum_{x=1}^N\psi^\Delta(x-1)e(\alpha\psi(x))\bigg|^{\rho}d\alpha\leq\psi(N)^{\rho-2L}\int_\M\bigg|\sum_{x=1}^N\psi^\Delta(x-1)e(\alpha\psi(x))\bigg|^{2L}d\alpha.
\end{align*}
So it suffices to show that
$$
\int_\M\bigg|\sum_{x=1}^N\psi^\Delta(x-1)e(\alpha\psi(x))\bigg|^{2L}d\alpha\ll_L\psi(N)^{2L-1}.
$$
Now
\begin{align*}
&\bigg|\sum_{x=1}^N\psi^\Delta(x-1)e(\alpha\psi(x))\bigg|^{2L}\\
=&\bigg|\frac{1}{q}\sum_{r=1}^qe(a\psi(r)/q)\sum_{x=1}^{\psi(N)}e((\alpha-a/q)x)\bigg|^{2L}+O(\psi(N)^{2L-1}\psi^\Delta(N)N^{\nu}).
\end{align*}
Hence
\begin{align*}
&\int_\M\bigg|\sum_{x=1}^N\psi^\Delta(x-1)e(\alpha\psi(x))\bigg|^{2L}d\alpha\\
=&\sum_{\substack{1\leq a\leq q\leq
N^\nu\\(a,q)=1}}\int_{\M_{a,q}}\bigg|\frac{1}{q}\sum_{r=1}^qe(a\psi(r)/q)\sum_{x=1}^{\psi(N)}e((\alpha-a/q)x)\bigg|^{2L}d\alpha\\
&+O(\psi(N)^{2L-1}\psi^\Delta(N)N^{\nu}\mes(\M)).
\end{align*}
Clearly
\begin{align*}
\int_{\M_{a,q}}\bigg|\sum_{x=1}^{\psi(N)}e((\alpha-a/q)x)\bigg|^{2L}d\alpha\leq&\int_{\T}\bigg|\sum_{x=1}^{\psi(N)}e((\alpha-a/q)x)\bigg|^{2L}d\alpha\\
=&\sum_{\substack{1\leq x_1,\ldots,x_{2L}\leq \psi(N)\\\
x_1+\cdots+x_L=x_{L+1}+\cdots+x_{2L}}}1\\
\leq&\psi(N)^{2L-1}.
\end{align*}
And by Lemma \ref{huasum},
$$
\sum_{\substack{1\leq a\leq q\leq
N^\nu\\(a,q)=1}}\bigg|\frac{1}{q}\sum_{r=1}^qe(a\psi(r)/q)\bigg|^{2L}\ll_\epsilon
\sum_{\substack{1\leq a\leq q\leq
N^\nu\\(a,q)=1}}q^{-2L(\frac{1}{k}-\epsilon)}\leq
\sum_{\substack{1\leq q\leq
N^\nu}}q^{1-2L(\frac{1}{k}-\epsilon)}=O_L(1)
$$
since $L>(\frac{1}{k}-\epsilon)^{-1}$. We are done.
\end{proof}

\begin{Lem}
\label{huasumprime}
$$
\sum_{\substack{1\leq r\leq q\\
(Wr+b,q)=1}}e(a\psi(r)/q)\ll_\epsilon\gcd(\psi)q^{1-\frac{1}{k(k+1)}+\epsilon}.
$$
\end{Lem}
\begin{proof}
Clearly
\begin{align*}
\sum_{\substack{1\leq r\leq q\\
(Wr+b,q)=1}}e(a\psi(r)/q)=\sum_{r=1}^qe(a\psi(r)/q)\sum_{d\mid
(Wr+b,q)}\mu(d)
\end{align*}
where $\mu$ is the M\"obius function. Note that $d\mid
(Wr+b)\Longrightarrow(d,W)=1$ since $(W,b)=1$. Hence
\begin{align*}
\sum_{\substack{1\leq r\leq q\\
(Wr+b,q)=1}}e(a\psi(r)/q)=\sum_{\substack{d\mid q\\ b_d\text{
exists}}}\mu(d)\sum_{\substack{1\leq r\leq q\\ r\equiv
b_d\pmod{d}}}e(a\psi(r)/q),
\end{align*}
where $1\leq b_d\leq d$ is the integer such that
$Wb_d+b\equiv0\pmod{d}$.

For those $d\leq q^{\frac{1}{k(k+1)}}$ with $b_d$ exists, we have
$$
\sum_{\substack{1\leq r\leq q\\ r\equiv
b_d\pmod{d}}}e(a\psi(r)/q)=\sum_{r=0}^{q/d-1}e(a\psi(dr+b_d)/q).
$$
Write
\begin{align*}
\psi(dr+b_d)=&\sum_{i=1}^ka_{k-i+1}\sum_{j=0}^i\binom{i}{j}d^jr^jb_d^{i-j}\\
=&\sum_{j=0}^kd^jr^j\sum_{i=j}^k\binom{i}{j}a_{k-i+1}b_d^{i-j}\\
=&a_1'r^k+a_2'r^{k-1}+\cdots+a_k'r+a_{k+1}'.
\end{align*}
Notice that
\begin{align*}
(q,a_1',\ldots,a_k')=(q,d^ka_1,a_2',\ldots,a_k')\leq
d^k(q,a_1,a_2',\ldots,a_k').
\end{align*}
Also
$$
a_2'=d^{k-1}(a_2+ka_1b_d).
$$
Therefore
$$
(q,a_1,a_2',\ldots,a_k')=(q,a_1,d^{k-1}a_2,\ldots,a_k')\leq
d^{k-1}(q,a_1,a_2,\ldots,a_k').
$$
Similarly, we obtain that
$$
(q,a_1',\ldots,a_k')\leq d^{\frac{k(k+1)}{2}}(q,a_1,\ldots,a_k).
$$
Thus by Lemma \ref{huasum},
\begin{align*}
\sum_{r=0}^{q/d-1}e(a\psi(dr+b_d)/q)\ll&_\epsilon(q/d,a_1',\ldots,a_k')\bigg(\frac{q/d}{(q/d,a_1',\ldots,a_k')}\bigg)^{1-\frac{1}{k}+\frac{\epsilon}{k}}\\
\leq&(q,a_1',\ldots,a_k')^{\frac{1-\epsilon}{k}}d^{\frac{1-\epsilon}{k}-1}q^{1-\frac{1-\epsilon}{k}}\\
\leq&(a_1,\ldots,a_k)^{\frac{1-\epsilon}{k}}d^{(\frac{k+1}{2}+\frac{1}{k})(1-\epsilon
)-1}q^{1-\frac{1-\epsilon}{k}}.
\end{align*}
On the other hand, clearly
$$
\bigg|\sum_{r=0}^{q/d-1}e(a\psi(dr+b_d)/q)\bigg|\leq
\frac{q}{d}<q^{1-\frac{1}{k(k+1)}}
$$
when $d>q^{\frac{1}{k(k+1)}}$.
Thus
\begin{align*}
&\bigg|\sum_{\substack{1\leq r\leq q\\ (Wr+b,q)=1}}e(a\psi(r)/q)\bigg|\\
\leq&\sum_{\substack{d\mid q,\ d\leq q^{\frac{1}{k(k+1)}}\\\text{
and }b_d\text{ exists}}}\bigg|\sum_{\substack{1\leq r\leq q\\
r\equiv b_d\pmod{d}}}e(a\psi(r)/q)\bigg|
+\sum_{\substack{d\mid q,\ d>q^{\frac{1}{k(k+1)}}\\\text{ and }b_d\text{ exists}}}\bigg|\sum_{\substack{1\leq r\leq q\\ r\equiv b_d\pmod{d}}}e(a\psi(r)/q)\bigg|\\
\ll&_\epsilon d(q)(\gcd(\psi)^{\frac{1-\epsilon}{k}}q^{1-\frac{1-\epsilon}{k}+\frac{1-\epsilon}{2(k+1)}}+q^{1-\frac{1}{k(k+1)}})\\
\ll&_\epsilon\gcd(\psi)q^{1-\frac{1}{k(k+1)}+\epsilon},
\end{align*}
where $d(q)$ is the divisor function.
\end{proof}

\begin{Lem}
\label{primepolyminor} For any  $A>0$, there is a $B=B(A,k)>0$
such that,
$$
\sum_{x=1}^N\lambda_{b,W}(x)e(\alpha\psi(x))\ll_B N(\log N)^{-A}
$$
provided that $|\alpha-a/q|\leq q^{-2}$ with $1\leq a\leq q$,
$(a,q)=1$ and $(\log N)^B\leq q\leq\psi(N)(\log N)^{-B}$.
\end{Lem}
\begin{proof}
At least Vinogradov had dealt with the case $\psi(x)=x^k$ and
$W=1$ in \cite{Vinogradov76}. The proof of this Lemma is very
standard but too long, so we give the detailed proof as an appendix.
\end{proof}

\begin{Lem}\
\label{huaprimepoly}
$$
\int_\T\bigg|\sum_{x=1}^N\psi^\Delta(x-1)\lambda_{b,W}(x)e(\alpha\psi(x))\bigg|^\rho
d\alpha\ll_\rho\gcd(\psi)\psi(N)^{\rho-1}
$$
for $\rho\geq k2^{k+2}+1$.
\end{Lem}
\begin{proof}
Without loss of generality, we assume that $\gcd(\psi)=1$. Let
$B>2\rho$ be a sufficiently large integer satisfying the requirement
of Lemma \ref{primepolyminor} for $A=2\rho$. Let
$$
\M_{a,q}=\{\alpha\in\T:\,|\alpha q-a|\leq (\log
N)^{2B}/\psi(N)\}, \quad\M=\bigcup_{\substack{1\leq a\leq q\leq
(\log N)^{2B}\\(a,q)=1}}\M_{a,q}
$$
and $\m=\T\setminus\M$.

If $\alpha\in\m$, then there exist $(\log N)^{2B}\leq
q\leq\psi(N)(\log N)^{-2B}$ and $1\leq a\leq q$ with $(a,q)=1$
such that $|\alpha-a/q|\leq q^{-2}$. By Lemma \ref{primepolyminor},
\begin{align*}
&\sum_{x=1}^y\lambda_{b,W}(x)e(\alpha\psi(x))\ll_By(\log
y)^{-2\rho}.
\end{align*}
for $N(\log N)^{-\frac{B}{k}}\leq y\leq N$. Therefore
\begin{align*}
&\bigg|\sum_{x=1}^N\psi^\Delta(x-1)\lambda_{b,W}(x)e(\alpha\psi(x))\bigg|\\
=&\bigg|\psi^\Delta(N-1)\sum_{x=1}^Ne(\alpha\psi(x))\lambda_{b,W}(x)-
\sum_{y=1}^{N-1}(\psi^\Delta)^\Delta(y-1)\sum_{x=1}^ye(\alpha\psi(x))\lambda_{b,W}(x)\bigg|\\
\leq&\psi^\Delta(N-1)\bigg|\sum_{x=1}^Ne(\alpha\psi(x))\lambda_{b,W}(x)\bigg|+\sum_{1\leq y<N(\log N)^{-\frac{B}{k}}}|(\psi^\Delta)^\Delta(y-1)y|\\
&+\sum_{N(\log N)^{-\frac{B}{k}}\leq y<N}(\psi^\Delta)^\Delta(y-1)\bigg|\sum_{x=1}^ye(\alpha\psi(x))\lambda_{b,W}(x)\bigg|\\
\ll&_{B}\psi(N)(\log N)^{-2\rho}.
\end{align*}
Let $L=\floor{(\rho-1)/2}$, then we have
\begin{align*}
&\int_{\m}\bigg|\sum_{x=1}^N\psi^\Delta(x-1)\lambda_{b,W}(x)e(\alpha\psi(x))\bigg|^{\rho}d\alpha\\
\ll&_B (\psi(N)(\log
N)^{-2\rho})^{\rho-2L}\int_{\m}\bigg|\sum_{x=1}^N\psi^\Delta(x-1)\lambda_{b,W}(x)e(\alpha\psi(x))\bigg|^{2L}d\alpha\\
\ll&_L \psi(N)^{\rho-2L}(\log
N)^{-2\rho}\int_{\T}\bigg|\sum_{x=1}^N\psi^\Delta(x-1)\lambda_{b,W}(x)e(\alpha\psi(x))\bigg|^{2L}d\alpha
\end{align*}
Noting that
\begin{align*}
&\int_{\T}\bigg|\sum_{x=1}^N\psi^\Delta(x-1)\lambda_{b,W}(x)e(\alpha\psi(x))\bigg|^{2L}d\alpha\\
=&\sum_{\substack{1\leq x_1,\ldots,x_{2L}\leq
N\\\psi(x_1)+\cdots+\psi(x_L)=\psi(x_{L+1})+\cdots+\psi(x_{2L})}}
\prod_{j=1}^{2L}\psi^\Delta(x_j-1)\lambda_{b,W}(x_j)\\
\leq&(\log(WN+b))^{2L}\sum_{\substack{1\leq x_1,\ldots,x_{2L}\leq
N\\\psi(x_1)+\cdots+\psi(x_L)=\psi(x_{L+1})+\cdots+\psi(x_{2L})}}
\prod_{j=1}^{2L}\psi^\Delta(x_j-1)\\
\ll&_L(\log N)^{2L}\int_{\T}\bigg|\sum_{x\leq
N}\psi^\Delta(x-1)e(\alpha\psi(x))\bigg|^{2L}d\alpha,
\end{align*}
so using Lemma \ref{huapoly} we have
\begin{align*}
\int_{\m}\bigg|\sum_{x\leq
N}\psi^\Delta(x-1)\lambda_{b,W}(x)e(\alpha\psi(x))\bigg|^{\rho}d\alpha\ll_L\psi(N)^{\rho-1}(\log
N)^{-\rho}.
\end{align*}

If $\alpha\in\M_{a,q}$, then by Lemma \ref{primepolymajor}
\begin{align*}
&\bigg|\sum_{x\leq N}\psi^\Delta(x-1)\lambda_{b,W}(x)e(\alpha\psi(x))\bigg|^{\rho}\\
=&\bigg|\frac{\phi(W)}{\phi(Wq)}\sum_{\substack{1\leq r\leq
q\\(Wr+b,q)=1}}e(a\psi(r)/q)\sum_{x\leq
N}\psi^\Delta(x-1)e((\alpha-a/q)\psi(x))\bigg|^{\rho}\\
&+O(\psi(N)^{\rho}(\log N)^{-7B}).
\end{align*}
In view of Lemma \ref{huasumprime}, letting $\epsilon=(k+2)^{-4}$,
\begin{align*}
\sum_{\substack{1\leq a\leq q\leq (\log N)^B\\
(a,q)=1}}\bigg|\frac{\phi(W)}{\phi(Wq)}\sum_{\substack{1\leq r\leq
q\\(Wr+b,q)=1}}e(a\psi(r)/q)\bigg|^{\rho} \ll_\epsilon\sum_{1\leq
q\leq(\log
N)^B}q^{1-\rho(\frac{1}{k(k+1)}-2\epsilon)}=O_{\rho,\epsilon}(1).
\end{align*}
Applying Lemma \ref{huapoly}, we concludes that
\begin{align*}
&\int_{\M}\bigg|\sum_{x\leq N}\psi^\Delta(x-1)\lambda_{b,W}(x)e(\alpha\psi(x))\bigg|^{\rho}d\alpha\\
=&
\sum_{\substack{1\leq a\leq q\leq (\log N)^B\\
(a,q)=1}}\bigg|\frac{\phi(W)}{\phi(Wq)}\sum_{\substack{1\leq r\leq
q\\(Wr+b,q)=1}}e(a\psi(r)/q)\bigg|^{\rho}\\
&\cdot\int_{\M_{a,q}}\bigg|\sum_{x\leq
N}\psi^\Delta(x-1)e((\alpha-a/q)\psi(x))\bigg|^{\rho}d\alpha+O(\mes(\M)\psi(N)^{\rho}(\log N)^{-7B})\\
\leq&
\bigg(\sum_{\substack{1\leq a\leq q\leq (\log N)^B\\
(a,q)=1}}\bigg|\frac{\phi(W)}{\phi(Wq)}\sum_{\substack{1\leq r\leq
q\\(Wr+b,q)=1}}e(a\psi(r)/q)\bigg|^{\rho}\bigg)\\
&\cdot\int_{\T}\bigg|\sum_{x\leq
N}\psi^\Delta(x-1)e(\alpha\psi(x))\bigg|^{\rho}d\alpha+O(\psi(N)^{\rho-1}(\log N)^{-B})\\
\ll&_{\rho,\epsilon}\psi(N)^{\rho-1}.
\end{align*}
\end{proof}

\begin{Lem}
\label{primepolyrestriction}
Suppose that $\psi$ is positive and
strictly increasing on $[1,N]$. Let $p\geq\psi(N)$ be a prime. Then
$$
\frac{1}{p}\sum_{r=1}^{p}\bigg|\sum_{z=1}^N\psi^\Delta(z-1)\lambda_{b,W}(z)e(-r\psi(z)/p)\bigg|^\rho
\ll_\rho\gcd(\psi)\psi(N)^{\rho-1}
$$
for $\rho\geq k2^{k+2}+1$.
\end{Lem}
\begin{proof}
We require a well-known result of Marcinkiewicz and Zygmund (cf.
\cite[Lemma 6.5]{Green05}):
$$
\sum_{r\in\Z_{p}}\bigg|\sum_{x=1}^{p}f(x)e(-xr/p)\bigg|^\rho\ll_\rho
p\int_\T |\hat{f}(\theta)|^\rho d\theta
$$
for arbitrary function $f:\,\Z_{p}=\Z/p\Z\to\C$,  where
$$
\hat{f}(\theta)=\sum_{x=1}^{p}f(x)e(-\theta x).
$$
Define
$$
f(x)=\begin{cases} \psi^\Delta(z-1)\lambda_{b,
W}(z)&\text{if }x=\psi(z)\text{ where }1\leq z\leq N,\\
0&\text{otherwise}.
\end{cases}
$$
Then
\begin{align*}
&\sum_{r\in\Z_{p}}\bigg|\sum_{z=1}^N\psi^\Delta(z-1)\lambda_{b,W}(z)e(-\psi(z)r/{p})\bigg|^\rho\\
=&\sum_{r\in\Z_{p}}\bigg|\sum_{x=1}^{p}f(x)e(-xr/{p})\bigg|^\rho\\
\ll&_\rho{p}\int_\T\bigg|\sum_{x=1}^{p}f(x)e(-x\theta)\bigg|^\rho d\theta\\
=&{p}\int_\T\bigg|\sum_{z=1}^N\psi^\Delta(z-1)\lambda_{b,W}(z)e(-\psi(z)\theta)\bigg|^\rho d\theta\\
\ll&_\rho\gcd(\psi){p}\psi(N)^{\rho-1},
\end{align*}
where Lemma \ref{huaprimepoly} is applied in the last inequality.
\end{proof}

\section{Proof of Theorem \ref{diffprimepoly}}
\setcounter{Lem}{0}\setcounter{Thm}{0}\setcounter{Cor}{0}
\setcounter{equation}{0}

Clearly Theorem \ref{diffprimepoly} is a consequence of the
following theorem:
\begin{Thm}
\label{diffprimepoly2} Suppose that $k\geq t\geq 1$ are integers,
$a_{k-t+1}$ is a non-zero integer and  $0<\delta\leq 1$. Let
$\psi(x)=a_1x^k+a_2x^{k-1}+\cdots+a_{k-t+1}x^t$ be an arbitrary
polynomial with integral coefficients and positive leading
coefficient. Then for any positive integer $W$, there exist
$N(\delta,W,\psi)$ and $c(\delta,a_{k-t+1})>0$ satisfying that
$$
\min_{\substack{A\subseteq\{1,2,\ldots,n\}\\|A|\geq\delta
n}}|\{(x,y,z):\, x,y\in A, z\in\Lambda_{1,W}, x-y=\psi(z)\}|\geq
c(\delta,a_{k-t+1})\frac{Wn^{1+\frac{1}{k}}a_1^{-\frac{1}{k}}}{\phi(W)\log
n}
$$
if $n\geq N(\delta,W,\psi)$.
\end{Thm}
\begin{Rem}
We emphasize that in Theorem \ref{diffprimepoly2} the constant $c(\delta,a_{k-t+1})$ only depends on $k, \delta, a_{k-t+1}$. As we will see later,
this fact is important in the proof of Theorem \ref{primediffprimepoly}.
\end{Rem}
\begin{proof}
Similarly as Tao's arguments \cite{Tao} on Roth's theorem
\cite{Roth53}, we shall make an induction on $\delta$. Suppose
that $P(\delta)$ is a proposition on $0<\delta\leq 1$. Assume that
$P(\delta)$ satisfies the following conditions:

\medskip\noindent (i) There exists $0<\delta_0<1$ such that
$P(\delta)$ holds for any $\delta_0\leq\delta\leq1$.

\medskip\noindent (ii) There exists a continuous function $\epsilon(\delta)>0$ such that
$\delta+\epsilon(\delta)\leq 1$ for any $0<\delta\leq\delta_0$ and
$P(\delta+\epsilon(\delta))$ holds implies $P(\delta)$ also holds.

\medskip\noindent (iii) If $0<\delta'<\delta\leq 1$, then
$P(\delta')$ holds implies that $P(\delta)$ also holds.

\medskip\noindent
Then we claim that $P(\delta)$ holds for any $0<\delta\leq1$. In
fact, assume on the contrary that there exists $0<\delta\leq 1$
such that $P(\delta)$ doesn't hold. Let
$$
\delta^*=\limsup_{\substack{0<\delta\leq 1\\ P(\delta)\text{
doesn't hold}}}\delta.
$$
From the condition (i), we know that $\delta^*\leq\delta_0$. Since
$\delta+\epsilon(\delta)$ is continuous, there exists
$0<\delta_1<\delta^*$ such that
$$
|(\delta^*+\epsilon(\delta^*))-(\delta_1+\epsilon(\delta_1))|<\frac{1}{2}\epsilon(\delta^*),
$$
i.e., $0<\delta_1<\delta^*<\delta_1+\epsilon(\delta_1)\leq 1$. Hence
$P(\delta_1+\epsilon(\delta_1))$ holds but $P(\delta_1)$ doesn't
hold by the definition of $\delta^*$. This is obviously leads to a
contradiction with the conditions (ii) and (iii).

Suppose that $A$ is subset of $\{1,2,\ldots,n\}$ with
$|A|\geq\delta n$. Firstly, we shall show that Theorem
\ref{diffprimepoly2} holds for $\delta\geq 3/4$. Define
$$
r_{W,\psi}(A)=|\{(x,y,z):\, x,y\in A, z\in\Lambda_{1,W},
x-y=\psi(z)\}|.
$$
Clearly
$$
|\{z\in\Lambda_{1,W}:\,1\leq\psi(z)\leq n/3\}|\geq
\frac{1}{4k}\frac{Wn^{\frac{1}{k}}a_1^{-\frac{1}{k}}}{\phi(W)\log
n},
$$
whenever $n$ is sufficiently large (depending on the coefficients
of $\psi$). And for any $1\leq z\leq n/3$,
\begin{align*}
|\{(x,y):\,x,y\in A,x-y=z\}|=&|A\cap(z+A)|\\
=&2|A|-|A\cup(z+A)|\\
\geq&\frac{2\cdot3n}{4}-\frac{4n}{3}=\frac{n}{6}.
\end{align*}
Hence
$$
r_{W,\psi}(A)\geq
\frac{1}{24k}\frac{Wn^{1+\frac{1}{k}}a_1^{-\frac{1}{k}}}{\phi(W)\log
n}.
$$

Now we assume that $\delta<3/4$. Let
$\epsilon=\epsilon(\delta,a_{k-t+1})$ be a small positive real
number and $Q=Q(\delta,a_{k-t+1})$ be a large integer to be chosen
later. We shall show that if Theorem \ref{diffprimepoly2} holds
for $\delta+\epsilon$, it also holds for $\delta$. Define
$$
\psi_{q}(x)=\psi(qx)/q^t=a_1q^{k-t}x^k+\cdots+a_{k-t+1}x^t.
$$
By the induction hypothesis on
$\delta+\epsilon$, for any $1\leq q\leq Q$
$$
\min_{\substack{A\subseteq\{1,2,\ldots,n\}\\|A|\geq(\delta+\epsilon)n}}r_{Wq,\psi_{q}}(A)\geq
\frac{c(\delta+\epsilon,a_{k-t+1})}{2} \frac{Wq}{\phi(Wq)}
\frac{n^{1+\frac{1}{k}}(a_1q^{k-t})^{-\frac{1}{k}}}{\log n}
$$
provided that
$$
n\geq\max_{1\leq q\leq Q}N(\delta+\epsilon,Wq,\psi_q).
$$

Let $\A_m(b,d)$ denote the arithmetic progression
$\{b,b+d,\ldots,b+(m-1)d\}$. Suppose that
$$
n\geq \max\{e^{k(|a_1|+\cdots+|a_{k-t+1}|)Q^{k-t}},
10^4\epsilon^{-1}Q^t\max_{1\leq q\leq
Q}N(\delta+\epsilon,Wq,\psi_q)\}
$$
and $A\subseteq\{1,2,\ldots,n\}$ with $|A|=\delta n$. Let
$m=\floor{10^{-2}\epsilon Q^{-t}n}$. Observe that $|\{b:\,x,y\in
\A_m(b,q^t)\}|\leq m$ for every pair $(x,y)$. Let
$$
A_{b,q^t}=\{1+(x-b)/q^t\:,x\in A\cap
\A_m(b,q^t)\}\subseteq\{1,2,\ldots,m\}.
$$
Clearly if $x',y'\in A_{b,q^t}$ and $z'\in\Lambda_{1,Wq}$ satisfy
that $x'-y'=\psi_q(z')$, then
$$
x=b+(x'-1)q^t,\  y=b+(y'-1)q^t\in A,\ z=z'q\in\Lambda_{1,W}
$$
and $x-y=\psi(z)$. So if there exists $1\leq q\leq Q$ such that
$$
|\{1\leq b\leq n-mq^t:\, |A_{b,q^t}|\geq(\delta+\epsilon)m\}|\geq
\epsilon n,
$$
then
\begin{align*}
r_{W,\psi}(A)\geq&\frac{1}{m}\sum_{1\leq b\leq
n-mq^t}r_{Wq,\psi_{q}}(A_{b,q^t})\\
\geq&\epsilon n\frac{c(\delta+\epsilon,a_{k-t+1})}{2}
\frac{Wq}{\phi(Wq)}
\frac{m^{\frac{1}{k}}(a_1q^{k-t})^{-\frac{1}{k}}}{\log m}\\
\geq&\frac{c(\delta+\epsilon,a_{k-t+1})\epsilon^{1+\frac{1}{k}}}{400Q}\cdot
\frac{Wn^{1+\frac{1}{k}}a_1^{-\frac{1}{k}}}{\phi(W)\log n}.
\end{align*}

So we may assume that
\begin{align}
\label{assumption} |\{1\leq b\leq n-mq^t:\, |A\cap
\A_m(b,q^t)|\geq(\delta+\epsilon)m\}|<\epsilon n
\end{align}
for each $1\leq q\leq Q$. Let
$$
M=\max\{x\in\Z:\,\psi(x)\leq n\}.
$$
Clearly $M=n^{\frac{1}{k}}a_1^{-\frac{1}{k}}(1+o(1))$. We shall
show that
$$
\int_\T\bigg(\bigg|\sum_{x\leq n}\1_Ae(\alpha
x)\bigg|^2-\delta^2\bigg|\sum_{x\leq n}e(\alpha
x)\bigg|^2\bigg)\bigg(\sum_{z\leq
M}\psi^\Delta(z-1)\lambda_{1,W}(z)e(\alpha \psi(z))\bigg)d\alpha
$$
is relatively small.

For $1\leq q\leq Q$, define
$$
\M_{a,q}=\{\alpha:\,|\alpha-a/q|\leq \frac{1}{2}q^{-t}m^{-1}\}.
$$
Let
$$
\M=\bigcup_{\substack{1\leq a\leq q\leq Q\\(a,q)=1}}\M_{a,q},
$$
and let $\m=\T\setminus\M$. Let $B$ be a sufficiently large integer.
For $1\leq q\leq(\log M)^B$, define
$$
\M_{a,q}^*=\{\alpha:\,|\alpha q-a|\leq(\log M)^B/\psi(M)\}.
$$
Let
$$
\M^*=\bigcup_{\substack{1\leq a\leq q\leq (\log
M)^B\\(a,q)=1}}\M_{a,q}^*
$$
and let $\m^*=\T\setminus\M^*$.

Suppose that $\alpha\in\m$. We know $|\alpha q-a|\leq (\log
M)^B/\psi(M)$ for some $1\leq a\leq q<\psi(M)(\log M)^{-B}$ with
$(a,q)=1$.
If $\alpha\in\m^*$, i.e., $q\geq(\log M)^B$, then
$|\alpha-a/q|\leq q^{-2}$ and $(\log y)^\frac{B}{2}\leq \psi(y)(\log y)^{-\frac{B}{2}}$
for any $M(\log M)^{-\frac{B}{2k}}\leq y\leq M$.
So applying Lemma \ref{primepolyminor} and
a partial summation, we have
\begin{align*}
\sum_{z\leq
M}\psi^\Delta(z-1)\lambda_{1,W}(z)e(\alpha\psi(z))\ll_B\psi(M)(\log
M)^{-1}\leq n(\log M)^{-1},
\end{align*}
whenever $B$ is sufficiently large.

Now suppose that $q<(\log M)^B$, i.e., $\alpha\in\M^*$. Applying
Lemmas \ref{primepolymajor} and \ref{polydelta}, we have
\begin{align*}
&\sum_{z\leq
M}\psi^\Delta(z-1)\lambda_{1,W}(z)e(\alpha\psi(z))\\
=&\frac{\phi(W)}{\phi(Wq)}\sum_{\substack{1\leq r\leq
q\\(Wr+1,q)=1}}e(a\psi(r)/q)\sum_{z\leq
M}\psi^\Delta(z-1)e((\alpha-a/q)\psi(z))\\
&+O(\psi^\Delta(M)M(\log
M)^{-4B})\\
=&\frac{\phi(W)}{\phi(Wq)}\sum_{\substack{1\leq r\leq
q\\(Wr+1,q)=1}}e(a\psi(r)/q)\sum_{z\leq
n}e((\alpha-a/q)z)+O(\psi^\Delta(M)M(\log M)^{-4B}).
\end{align*}
Since $\alpha\in\m$, either $q>Q$ or
$|\alpha-a/q|>\frac{1}{2}q^{-t}m^{-1}$.

If $q>Q$, then in light of Lemma \ref{huasumprime}
\begin{align*}
\bigg|\frac{\phi(W)}{\phi(Wq)}\sum_{\substack{1\leq r\leq q\\
(Wr+1,q)=1}}e(a\psi(r)/q)\bigg|\leq
\bigg|\frac{1}{\phi(q)}\sum_{\substack{1\leq r\leq q\\
(Wr+1,q)=1}}e(a\psi(r)/q)\bigg|\leq
C_1|a_{k-t+1}|q^{-\frac{1}{k(k+2)}}.
\end{align*}
And if $|\alpha-a/q|>\frac{1}{2}q^{-t}m^{-1}$, then
$$
\bigg|\sum_{z=1}^{n}e((\alpha-a/q)z)\bigg|=\bigg|\frac{1-e((\alpha-a/q)n)}{1-e(\alpha-a/q)}\bigg|\leq
4\pi q^tm.
$$
Hence for $\alpha\in\m$
\begin{align*}
\sum_{z\leq M}\psi^\Delta(z-1)\lambda_{1,W}(z)e(\alpha\psi(z)) \leq
C_1|a_{k-t+1}|Q^{-\frac{1}{k(k+2)}}n+4\pi mQ^t+O(n(\log
n)^{-1}).
\end{align*}

Suppose that $\alpha\in\M$. Let $\tau=\1_A-\delta$ where $\1_A(x)=1$
or $0$ according whether $x\in A$ or not. Let
$$
S(\alpha)=\sum_{c=0}^{m-1}e(\alpha c)
$$
and
$$
T(\alpha)=\sum_{b=1}^n\tau(b)e(\alpha b).
$$
Then
\begin{align*}
S(\alpha q^t)T(\alpha)=&\sum_{b=1}^n\tau(b)\sum_{c=0}^{m-1}e(\alpha(b+cq^t))\\
=&\sum_{b=1}^{n-mq^t}e(\alpha(b+(m-1)q^t))\sum_{c=0}^{m-1}\tau(b+cq^t)+R(\alpha),
\end{align*}
where $|R(\alpha)|\leq 2m^2q^t$. When $|\alpha q^t-aq^{t-1}|\leq
\frac{1}{2}m^{-1}$,
$$
|S(\alpha q^t)|=|S(\alpha q^t-aq^{t-1})|=\bigg|\frac{1-e(m(\alpha
q^t-aq^{t-1}))}{1-e(\alpha q^t-aq^{t-1})}\bigg|\geq \frac{m}{\pi}.
$$
Hence for $\alpha\in\M_{a,q}$,
$$
m|T(\alpha)|\leq\pi|S(\alpha q^t)T(\alpha)|\leq
\pi\bigg|\sum_{b=1}^{n-mq^t}e(\alpha(b+(m-1)q^t))\sum_{c=0}^{m-1}\tau(b+cq^t)\bigg|+\pi|R(\alpha)|.
$$
Notice that
$$
|\{1\leq b\leq n-mq^t:\, x\in\A_m(b,q^t)\}|\leq m,
$$
and the equality holds if $1+(m-1)q^t\leq x\leq n-mq^t$. It
follows that
$$
m|A|\geq\sum_{b=1}^{n-mq^t}|A\cap\A_m(b,q^t)|=\sum_{x\in
A}\sum_{b=1}^{n-mq^t}\1_{\A_m(b,q^t)}(x)\geq m|A|-2m^2q^t,
$$
whence
$$
\bigg|\sum_{b=1}^{n-mq^t}(|A\cap\A_m(b,q^t)|-(\delta+\epsilon)
m)\bigg|\leq\epsilon nm+(2+\delta)m^2q^t.
$$
By the assumption (\ref{assumption}), we have
$$
\sum_{\substack{1\leq b\leq n-mq^t\\
|A\cap\A_m(b,q^t)|\geq (\delta+\epsilon) m}}(|A\cap
\A_m(b,q^t)|-(\delta+\epsilon) m)\leq \epsilon n(1-\delta)m.
$$
It follows that
\begin{align*}
\sum_{b=1}^{n-mq^t}||A\cap\A_m(b,q^t)|-\delta m|
\leq&\sum_{b=1}^{n-mq^t}||A\cap\A_m(b,q^t)|-(\delta+\epsilon) m|+\epsilon nm\\
\leq&2\sum_{\substack{1\leq b\leq n-mq^t\\
|A\cap\A_m(b,q^t)|\geq (\delta+\epsilon) m}}(|A\cap
\A_m(b,q^t)|-(\delta+\epsilon)
m)\\
&+\bigg|\sum_{b=1}^{n-mq^t}(|A\cap\A_m(b,q^t)|-(\delta+\epsilon) m)\bigg|+\epsilon nm\\
\leq&4\epsilon nm+4m^2q^t.
\end{align*}
Thus for any $\alpha\in\M$.
\begin{align*}
|T(\alpha)|\leq&\frac{\pi}{m}
\bigg(\bigg|\sum_{b=1}^{n-mq^t}e(\alpha(b+(m-1)q^t))\sum_{c=0}^{m-1}\tau(b+cq^t)\bigg|+2m^2q^t\bigg)\\
\leq&\frac{\pi}{m}
\bigg(\sum_{b=1}^{n-mq^t}||A\cap\A_m(b,q^t)|-\delta m|+2m^2q^t\bigg)\\
\leq&4\pi\epsilon n+6\pi mQ^t,
\end{align*} i.e.,
$$
\bigg|\sum_{x=1}^n\1_A(x)e(\alpha x)-\delta\sum_{x=1}^ne(\alpha
x)\bigg|\leq16\epsilon n.
$$

It is easy to see that
$$
||x|^2-|y|^2|\leq||x|-|y||^{\frac2\rho}(|x|+|y|)^{2-\frac{2}{\rho}}\leq
4|x-y|^{\frac2\rho}(|x|^{2-\frac{2}{\rho}}+|y|^{2-\frac{2}{\rho}})
$$
for any $\rho\geq 2$. Let $\rho=k2^{k+3}$. Then
\begin{align*}
&\bigg|\int_{\M}\bigg(\bigg|\sum_{x=1}^n\1_A(x)e(\alpha
x)\bigg|^2-\delta^2\bigg|\sum_{x=1}^ne(\alpha
x)\bigg|^2\bigg)\bigg(\sum_{z=1}^{M}\psi^\Delta(z-1)\lambda_{1,W}(z)e(\alpha
\psi(z))\bigg)d\alpha\bigg|\\
\leq&4(16\epsilon
n)^{\frac2\rho}\int_{\M}\bigg(\bigg|\sum_{x=1}^n\1_A(x)e(\alpha
x)\bigg|^{2-\frac2\rho}+\delta^{2-\frac2\rho}\bigg|\sum_{x=1}^ne(\alpha
x)\bigg|^{2-\frac2\rho}\bigg)\\
&\cdot\bigg|\sum_{z=1}^{M}\psi^\Delta(z-1)\lambda_{1,W}(z)e(\alpha
\psi(z))\bigg|d\alpha.
\end{align*}
By the H\"older inequality,
\begin{align*}
&\int_{\M}\bigg|\sum_{x=1}^n\1_A(x)e(\alpha
x)\bigg|^{2-\frac2\rho}\bigg|\sum_{z=1}^{M}\psi^\Delta(z-1)\lambda_{1,W}(z)e(\alpha
\psi(z))\bigg|d\alpha\\
\leq& \bigg(\int_{\T}\bigg|\sum_{x=1}^n\1_A(x)e(\alpha x)\bigg|^
2d\alpha\bigg)^{1-\frac1{\rho}}\bigg(\int_{\T}\bigg|\sum_{z=1}^{M}\psi^\Delta(z-1)\lambda_{1,W}(z)e(\alpha
\psi(z))\bigg|^{\rho}d\alpha\bigg)^{\frac1{\rho}}.
\end{align*}
Applying Lemma \ref{huaprimepoly},
\begin{align*}
\int_{\T}\bigg|\sum_{z=1}^{M}\psi^\Delta(z-1)\lambda_{1,W}(z)e(\alpha
\psi(z))\bigg|^{\rho}d\alpha \leq C_2|a_{k-t+1}|\psi(M)^{\rho-1}.
\end{align*}
Therefore
\begin{align*}
&\int_{\M}\bigg|\sum_{x=1}^n\1_A(x)e(\alpha
x)\bigg|^{2-\frac2\rho}\bigg|\sum_{z=1}^{M}\psi^\Delta(z-1)\lambda_{1,W}(z)e(\alpha
\psi(z))\bigg|d\alpha\\\leq&
C_2^{\frac1\rho}|a_{k-t+1}|^{\frac1\rho}|(\delta
n)^{1-\frac1{\rho}}n^{1-\frac1{\rho}}.
\end{align*}
Similarly,
\begin{align*}
\int_{\M}\bigg|\sum_{x=1}^ne(\alpha
x)\bigg|^{2-\frac{2}\rho}\bigg|\sum_{z=1}^{M}\psi^\Delta(z-1)\lambda_{1,W}(z)e(\alpha
\psi(z))\bigg|d\alpha \leq
C_2^{\frac1\rho}|a_{k-t+1}|^{\frac1\rho}n^{2-\frac2\rho}.
\end{align*}
It is concluded that
\begin{align*}
&\bigg|\int_{\M}\bigg(\bigg|\sum_{x=1}^n\1_A(x)e(\alpha
x)\bigg|^2-\delta^2\bigg|\sum_{x=1}^ne(\alpha
x)\bigg|^2\bigg)\bigg(\sum_{z=1}^{M}\psi^\Delta(z-1)\lambda_{1,W}(z)e(\alpha
\psi(z))\bigg)d\alpha\bigg|\\
\leq&8C_2^{\frac1\rho}|a_{k-t+1}|^{\frac1\rho}\epsilon^{\frac2\rho}(\delta^{1-\frac1{\rho}}+\delta^{2-\frac2\rho})
n^{2}.
\end{align*}

Now we have shown that
\begin{align*}
&\bigg|\int_\T\bigg(\bigg|\sum_{x\leq n}\1_Ae(\alpha
x)\bigg|^2-\delta^2\bigg|\sum_{x\leq n}e(\alpha
x)\bigg|^2\bigg)\bigg(\sum_{z\leq
M}\psi^\Delta(z-1)\lambda_{1,W}(z)e(\alpha \psi(z))\bigg)d\alpha\bigg|\\
\leq& (2C_1|a_{k-t+1}|Q^{-\frac{1}{k(k+2)}}n+5\pi
mQ^t)\int_\T\bigg(\bigg|\sum_{x\leq n}\1_Ae(\alpha
x)\bigg|^2+\delta^2\bigg|\sum_{x\leq n}e(\alpha
x)\bigg|^2\bigg)d\alpha\\
&+8C_2^{\frac1\rho}|a_{k-t+1}|^{\frac1\rho}\epsilon^{\frac2\rho}(\delta^{1-\frac1{\rho}}+\delta^{2-\frac2\rho})
n^{2}\\
\leq& 4C_1|a_{k-t+1}|Q^{-\frac{1}{k(k+2)}}\delta
n^2+\epsilon\delta
n^2+16C_2^{\frac1\rho}|a_{k-t+1}|^{\frac1\rho}\epsilon^{\frac2\rho}\delta^{1-\frac1{\rho}}
n^{2}.
\end{align*}
On the other hand, we have
\begin{align*}
&\int_{\T}\bigg|\sum_{x=1}^ne(\alpha
x)\bigg|^2\bigg(\sum_{z=1}^{M}\psi^\Delta(z-1)\lambda_{1,W}(z)e(\alpha
\psi(z))\bigg)d\alpha\\=&\sum_{\substack{1\leq x,y\leq n\\
1\leq z\leq M\\x-y=\psi(z)}}\psi^\Delta(z-1)\lambda_{1,W}(z)\\
\geq&\sum_{\substack{1\leq x,y\leq n\\
M/4+1\leq z\leq M/2\\x-y=\psi(z)}}\psi^\Delta(z-1)\lambda_{1,W}(z)\\
\geq&\frac{M}{8}(n-\psi(M/2))\psi^\Delta(M/4).
\end{align*}
It follows that
\begin{align*}
&\int_{\T}\bigg|\sum_{x=1}^n\1_{A}(x)e(\alpha
x)\bigg|^2\bigg(\sum_{z=1}^{M}\psi^\Delta(z-1)\lambda_{1,W}(z)e(\alpha
\psi(z))\bigg)d\alpha\\
\geq&\delta^2\int_{\T}\bigg|\sum_{x=1}^ne(\alpha
x)\bigg|^2\bigg(\sum_{z=1}^{M}\psi^\Delta(z-1)\lambda_{1,W}(z)e(\alpha
\psi(z))\bigg)d\alpha\\
&-4C_1|a_{k-t+1}|Q^{-\frac{1}{k(k+2)}}\delta n^2-\epsilon \delta
n^2-16C_2^{\frac1\rho}|a_{k-t+1}|^{\frac1\rho}\epsilon^{\frac{2}\rho}\delta^{1-\frac1{\rho}}n^2\\
\geq&\frac{k\delta^2
n^2}{4^{k+1}}-4C_1|a_{k-t+1}|Q^{-\frac{1}{k(k+2)}}\delta
n^2-\epsilon \delta
n^2-16C_2^{\frac1\rho}|a_{k-t+1}|^{\frac1\rho}\epsilon^{\frac{2}\rho}\delta^{1-\frac1{\rho}}n^2.
\end{align*}
Let
$\epsilon=4^{-(k+2)\rho}\delta^{\frac{\rho+1}{2}}C_2^{-\frac{1}{2}}|a_{k-t+1}|^{-\frac{1}{2}}$
and
$$
Q=4^{(k+1)^4}\delta^{-2k(k+2)}C_1^{k(k+2)}|a_{k-t+1}|^{k(k+2)}.
$$
Therefore
\begin{align*}
&|\{(x,y,z):\,x,y\in A, z\in\Lambda_{1,W}, x-y=\psi(z)\}|\\
\geq&\frac{W/\phi(W)}{\psi^\Delta(M)\log
(WM+1)}\int_{\T}\bigg|\sum_{x=1}^n\1_{A}(x)e(\alpha
x)\bigg|^2\bigg(\sum_{z=1}^{M}\psi^\Delta(z-1)\lambda_{1,W}(z)e(\alpha
\psi(z))\bigg)d\alpha\\
\geq&\frac{W\delta^2}{4^{k+2}k\phi(W)}\cdot\frac{
n^{1+\frac{1}{k}}a_1^{-\frac{1}{k}}}{\log n}.
\end{align*}
This concludes our desired result.
\end{proof}

Finally, let us briefly discuss the bound in Theorem \ref{diffprimepoly}. Let
$R_{W,\psi}(\delta)$ be the least integer $n$ such that for any
$A\subseteq\{1,2,\ldots,n\}$, there exist $x,y\in A$ and $z\in\Lambda_{1,W}$ satisfying $x-y=\psi(z)$.
In our proof, we choose
$\epsilon=\epsilon(\delta)=O_{|a_{k-t}|}(\delta^{O_k(1)})$ and $Q=Q(\delta)=O_{|a_{k-t}|}(\delta^{-O_k(1)})$. So the iteration process
$\delta\to\delta+\epsilon(\delta)$ will end after $O_{|a_{k-t}|}(\delta^{-O_k(1)})$ steps. Also, clearly for $\delta>3/4$,
$$
R_{W,\psi}(\delta)\ll (|a_1|+\cdots+|a_{k-t}|)(\min\{p:\,p\in\Lambda_{1,W}\})^k.
$$
Notice that when the iteration process ends, $W$ will become $WQ^{O_{|a_{k-t}|}(\delta^{-O_k(1)})}$ and $a_i$ will become $a_iQ^{O_{|a_{k-t}|}(\delta^{-O_k(1)})}$.
Hence we have
$$
R_{W,\psi}(\delta)\leq\exp(O_{W,a_1,\ldots,a_{k-t}}(\delta^{-O_{|a_{k-t}|}(\delta^{-O_{k}(1)})})),
$$
since $\min\{p:\,p\in\Lambda_{1,W}\}\leq e^{O(W)}$. In other words, if a subset $A\subseteq\{1,2,\ldots,n\}$ satisfies
$|A|\geq O_{W,a_1,\ldots,a_{k-t}}(n/\log\log\log n)$,
then there exist $x,y\in A$ and $z\in\Lambda_{1,W}$ such that $x-y=\psi(z)$. Of course, this bound is very rough. And we believe
that it could be improved using some more refined estimations (e.g. \cite{PintzSteigerSzemeredi88}, \cite{BalogPelikanPintzSzemeredi94}, \cite{Lucier06}, \cite{Lucier}, \cite{RuzsaSanders}).

\section{Proof of Theorem \ref{primediffprimepoly}}
\setcounter{Lem}{0}\setcounter{Thm}{0}\setcounter{Cor}{0}
\setcounter{equation}{0}

Write $\psi(x)=a_1x^k+a_2x^{k-1}+\cdots+a_{k-t+1}x^t$ where
$a_{k-t+1}\not=0$. Let $\delta=\od_\P(P)$. Since $\od_\P(P)>0$,
there exist infinitely many $n$ such that
$$
|P\cap[1, n]|\geq\frac{4\delta}{5}\cdot\frac{n}{\log n}.
$$
Define
$$
w(n)=\max\{w\leq \log\log\log n:\, n\geq
16\W(w)N(\delta,\W(w),\psi_{\W(w)})\},
$$
where $N(\delta,W,\psi)$ is same as the one defined in Theorem
\ref{diffprimepoly2} and $\W(w)=\prod_{\substack{p\leq w\\
p\text{ prime}}}p$. Clearly $\lim_{n\to\infty} w(n)=\infty$. Let
$w=w(n)$ and $\W =\W(w)$. Then
$$
\sum_{\substack{x\in P\cap[1,n]\\ (x,\W)=1}}\log
x\geq\sum_{\substack{x\in P\cap[n^{\frac23},n]}}\log x\geq
\frac{2\log n}{3}(|P\cap[1, n]|-n^{\frac{2}{3}})\geq
\frac{\delta}{2}\cdot n.
$$
Hence there exists $1\leq b\leq \W ^t$ with $(b,\W)=1$ such that
$$
\sum_{\substack{x\in P\cap[1,n]\\ x\equiv b\pmod{\W ^t}}}\log
x\geq \frac{\delta}{2\phi(\W ^t)}\cdot n.
$$
Let
$$
A=\{(x-b)/\W ^t:\,x\in P\cap[1,n],\ x\equiv b\pmod{\W ^t}\}.
$$
Let $N$ be a prime in the interval $(2n/\W ^t,4n/\W ^t]$. Define
$\lambda_{b,\W^t,N}=\lambda_{b,\W^t}/N$ and $a=\1_A\lambda_{b,\W
^t,N}$. Then
$$
\sum_{x}a(x)\geq \frac{\phi(\W ^t)}{\W ^tN}\cdot \frac{\delta
n}{2\phi(\W ^t)}\geq \frac{\delta}{8}.
$$
Let
$$
\psi_\W(x)=\psi(\W
x)/\W^t=a_1\W^{k-t}x^{k}+\cdots+a_{k-t+1}x^{t}.
$$
Clearly $\psi_\W(z)$
is positive and strictly increasing for $1\leq z\leq M$, whenever
$\W$ is sufficiently large.

Below we consider $A$ as a subset of $\Z_N$. Let $M=\max\{z\in
\mathbb{N} :\,\psi_{\W }(z)<N/2\}$. If $x,y\in A$ and $1\leq z\leq
M$ satisfy $x-y=\psi_{\W }(z)$ in $\Z_N$, then we also have
$x-y=\psi_{\W }(z)$ in $\Z$. In fact, since $1\leq x,y<N/2$ and
$1\leq z\leq M$, it is impossible that $x-y=\psi_{\W }(z)-N$ in
$\Z$. For a function $f:\,\Z_N\to\C$, define
$$
\tilde{f}(r)=\sum_{x\in\Z_N}f(x)e(-xr/N).
$$
\begin{Lem}[Bourgain \cite{Bourgain89}, \cite{Bourgain93} and Green \cite{Green05}]
\label{primerestriction} Suppose that $\rho>2$. Then
\begin{align*}
\sum_r|\tilde{a}(r)|^\rho\leq C(\rho),
\end{align*}
where $C(\rho)$ is a constant only depending on $\rho$.
\end{Lem}
\begin{proof} See \cite[Lemma 6.6]{Green05}.
\end{proof}
\begin{Lem}
\label{restrictionprimepoly}
\begin{align*}
\sum_{r\in\Z_N}\bigg|\sum_{z=1}^M\psi_{\W
}^\Delta(z-1)\lambda_{1,\W W}(z)e(-\psi_\W(z)r/N)\bigg|^\rho\leq
C'(\rho)|a_{k-t+1}|N^{\rho}.
\end{align*}
provided that $\rho\geq k2^{k+3}$, where $C'(\rho)$ is a constant
only depending on $\rho$.
\end{Lem}
\begin{proof}
This is an immediate consequence of Lemma
\ref{primepolyrestriction} since $\gcd(\psi_{\W})\leq |a_{k-t+1}|$.
\end{proof}

Let $\eta$ and $\epsilon$ be two positive real numbers to be
chosen later. Let
$$
R=\{r\in\Z_N:\,\tilde{a}(r)\geq\eta\}
$$
and
$$
B=\{r\in\Z_N:\,\|xr/N\|\leq\epsilon\text{ for all }r\in R\}.
$$
Define $\beta=\1_B/|B|$ and $a'=a*\beta*\beta$, where
$$
f*g(x)=\sum_{y\in\Z_N}f(y)g(x-y).
$$
Let $\varrho=k2^{k+3}$.
\begin{Lem}
$$
\sum_{\substack{x,y\in\Z_N\\ 1\leq z\leq M\\
x-y=\psi_{\W}(z)}}(a'(x)a'(y)-a(x)a(y))\psi_{\W
}^\Delta(z-1)\lambda_{1,\W W}(z)\leq
C(\epsilon^2\eta^{-\frac{5}{2}}+\eta^{\frac{1}{\varrho}}).
$$
\end{Lem}
\begin{proof}
It is not difficult to check that
\begin{align*}
&\sum_{\substack{x,y\in\Z_N\\ 1\leq z\leq M\\
x-y=\psi_{\W }(z)}}a(x)a(y)\psi_{\W}^\Delta(z-1)\lambda_{1,\W W}(z)\\
=&\frac1N
\sum_{\substack{r\in\Z_N}}\tilde{a}(r)\tilde{a}(-r)\bigg(\sum_{z=1}^M\psi_{\W
}^\Delta(z-1)\lambda_{1,\W W}(z)e(-\psi_{\W }(z)r/N)\bigg).
\end{align*}
Also, it is easy to see that
$(f*g)\,\tilde{}=\tilde{f}\tilde{g}$. Then
\begin{align*}
&\sum_{\substack{x,y\in\Z_N\\ 1\leq z\leq M\\
x-y=\psi_{\W }(z)}}a'(x)a'(y)\psi_{\W }^\Delta(z-1)\lambda_{1,\W
W}(z)-
\sum_{\substack{x,y\in\Z_N\\ 1\leq z\leq M\\
x-y=\psi_{\W }(z)}}a(x)a(y)\psi_{\W }^\Delta(z-1)\lambda_{1,\W W}(z)\\
=&\frac{1}{N}\sum_{\substack{r\in\Z_N}}\tilde{a}(r)\tilde{a}(-r)(\tilde{\beta}(r)^2\tilde{\beta}(-r)^2-1)\bigg(\sum_{z=1}^M\psi_{\W
}^\Delta(z-1)\lambda_{1,\W W}(z)e(-\psi_{\W }(z)r/N)\bigg).
\end{align*}
If $r\in R$, then by the proof of Lemma 6.7 of \cite{Green05}, we
know that
$$
|\tilde{\beta}(r)^2\tilde{\beta}(-r)^2-1|\leq 2^{16}\epsilon^2.
$$
And applying Lemma \ref{primepolymajor} with $\alpha=a=q=1$,
\begin{align*}
\sum_{z=1}^M\psi_{\W }^\Delta(z-1)\lambda_{1,\W
W}(z)=&\sum_{z=1}^M\psi_{\W }^\Delta(z-1)+O(\psi_{\W
}^\Delta(M)Me^{-c\sqrt{\log M}})\leq2\psi_{\W }(M).
\end{align*}
Therefore
\begin{align*}
&\bigg|\sum_{\substack{r\in R}}\tilde{a}(r)\tilde{a}(-r)(\tilde{\beta}(r)^2\tilde{\beta}(-r)^2-1)\bigg(\sum_{z=1}^M\psi_{\W }^\Delta(z-1)\lambda_{1,\W W}(z)e(-\psi_{\W }(z)r/N)\bigg)\bigg|\\
\leq&2^{16}\epsilon^2
\sum_{\substack{r\in R}}|\tilde{a}(r)|^2\cdot\bigg|\sum_{z=1}^M\psi_{\W }^\Delta(z-1)\lambda_{1,\W W}(z)e(-\psi_{\W }(z)r/N)\bigg|\\
\leq&2^{17}\epsilon^2\psi_{\W }(M)|R|.\\
\end{align*}
In view of Lemma \ref{primerestriction} with $\rho=5/2$, we have
$|R|\leq C''\eta^{-\frac{5}{2}}$. On the other hand, by H\"older
inequality,
\begin{align*}
&\bigg|\sum_{\substack{r\not\in R}}\tilde{a}(r)\tilde{a}(-r)(\tilde{\beta}(r)^2\tilde{\beta}(-r)^2-1)\bigg(\sum_{z=1}^M\psi_{\W }^\Delta(z-1)\lambda_{1,\W W}(z)e(-\psi_{\W }(z)r/N)\bigg)\bigg|\\
\leq&2\sup_{r\not\in R}|\tilde{a}(r)|^\frac{1}{\varrho}
\bigg(\sum_{\substack{r\not\in
R}}|\tilde{a}(r)|^\frac{2\varrho-1}{\varrho-1}\bigg)^{\frac{\varrho-1}{\varrho}}\bigg(\sum_{\substack{r\not\in
R}}\bigg|\sum_{z=1}^M\psi_{\W }^\Delta(z-1)\lambda_{1,\W W}(z)e(-\psi_{\W }(z)r/N)\bigg|^\varrho\bigg)^\frac1\varrho\\
\leq&2\eta^\frac{1}{\varrho}\cdot
C((2\varrho-1)/(\varrho-1))^{1-\frac{1}{\varrho}}\cdot(|a_{k-t+1}|C'(\varrho))^\frac1\varrho
N,
\end{align*}
where in the last step we apply Lemma \ref{primerestriction} with
$\rho=(2\varrho-1)/(\varrho-1)$ and Lemma
\ref{restrictionprimepoly} with $\rho=\varrho$.
All are done.
\end{proof}

\begin{Lem} If $\epsilon^{|R|}\geq 2\log\log w/w$, then
$|a'(x)|\leq2/N$ for any $x\in\Z_N$.
\end{Lem}
\begin{proof}
See \cite[Lemma 6.3]{Green05}.
\end{proof}

Let $A'=\{x\in\Z_N:\,a'(x)\geq\frac{1}{16}\delta N^{-1}\}$. Then
$$
\frac{2}{N}|A'|+\frac{\delta}{16N}(N-|A'|)\geq\sum_{x\in\Z_N}a'(x)=\sum_{x\in\Z_N}a(x)\geq\frac{\delta}{8},
$$
whence
$|A'|/N\geq\delta/32$.
Let $A_1'=A'\cap[1,(N-1)/2]$ and
$$
A_2'=\{x-(N-1)/2:\,x\in A'\cap[(N+1)/2,N-1]\}.
$$
Clearly there exists $i\in\{1,2\}$ such that
$|A_i'|/N\geq\delta/64$. Without loss of generality, we may assume
that $|A_1'|/N\geq\delta/64$. Applying Theorem
\ref{diffprimepoly2}, we know that
\begin{align*}
&|\{(x,y,z):\,x,y\in A_1', z\in\Lambda_{1,\W W}\cap[1,M],
x-y=\psi_\W(z)\}|\\
\geq& c(\delta/64,a_{k-t+1})\frac{\W W(N/2)^{1+\frac{1}{k}}(a_1\W
^{k-t})^{-\frac{1}{k}}}{\phi(\W W)\log N}.
\end{align*}
Let $c'=\frac{1}{16k}c(\delta/64,a_{k-t+1})$. Clearly
\begin{align*}
|\{(x,y,z):\,x,y\in A_1', z\in\Lambda_{1,\W W}\cap[1,c'M],
x-y=\psi_{\W }(z)\}|\leq\frac{\W W(c'M)}{\phi(\W W)\log M}N.
\end{align*}
Therefore
\begin{align*}
&|\{(x,y,z):\,x,y\in A_1', z\in\Lambda_{1,\W ^tW}\cap(c'M,M],
x-y=\psi_{\W }(z)\}|\\
\geq& \frac{c(\delta/64,a_{k-t+1})}{8}\frac{\W
WN^{1+\frac{1}{k}}(a_1\W ^{k-t})^{-\frac{1}{k}}}{\phi(\W W)\log
N}.
\end{align*}
It follows that
\begin{align*}
&\sum_{\substack{x,y\in A_1'\\1\leq z\leq
M\\ x-y=\psi_{\W }(z)}}\psi_{\W }^\Delta(z-1)\lambda_{1,\W W}(z)\\
\geq&\frac{c(\delta/64,a_{k-t+1})}{8}\frac{\W
WN^{1+\frac{1}{k}}(a_1\W ^{k-t})^{-\frac{1}{k}}}{\phi(\W W)\log
N}\cdot\frac{\psi_{\W }^\Delta(c'M)\phi(\W W)\log M}{2\W W}&\\
\geq&\frac{c(\delta/64,a_{k-t+1})c'^{k-1}}{64}N^2.
\end{align*}
So
\begin{align*}
&\sum_{\substack{x,y\in\Z_N\\ 1\leq z\leq M\\
x-y=\psi_{\W }(z)}}a(x)a(y)\psi_{\W }^\Delta(z-1)\lambda_{1,\W W}(z)\\
\geq&
\sum_{\substack{x,y\in\Z_N\\ 1\leq z\leq M\\
x-y=\psi_{\W }(z)}}a'(x)a'(y)\psi_{\W }^\Delta(z-1)\lambda_{1,\W W}(z)-C(\epsilon^2\eta^{-\frac{5}{2}}+\eta^{\frac{1}{\varrho}})\\
\geq&\frac{\delta^2}{2^{8}N^2}
\sum_{\substack{x,y\in A_1'\\ 1\leq z\leq M\\
x-y=\psi_{\W }(z)}}\psi_{\W }^\Delta(z-1)\lambda_{1,\W W}(z)-C(\epsilon^2\eta^{-\frac{5}{2}}+\eta^{\frac{1}{\varrho}})\\
\geq&c''(\delta,a_{k-t+1})-C(\epsilon^2\eta^{-\frac{5}{2}}+\eta^{\frac{1}{\varrho}}).
\end{align*}
Finally, we may choose $\eta, \epsilon>0$
satisfying $\epsilon^{C''\eta^{-5/2}}\geq 2\log\log w/w$ such that
$C(\epsilon^2\eta^{-\frac{5}{2}}+\eta^{\frac{1}{\varrho}})<c''(\delta,a_{k-t+1})/2$,
whenever $w$ is sufficiently large. Hence
\begin{align*}
\sum_{\substack{x,y\in\Z_N\\ 1\leq z\leq M\\
x-y=\psi_{\W }(z)}}a(x)a(y)\psi_{\W }^\Delta(z-1)\lambda_{1,\W
W}(z)\geq\frac{c''(\delta,a_{k-t+1})}{2}>0
\end{align*}
for sufficiently large $N$.
\qed

\section*{Appendix:\ Exponential Sums on Polynomials of Prime Variables}

\begin{Lem}
\begin{equation}
\sum_{x=1}^Nd_k^2(x)\ll N(\log N)^{k^2-1},
\end{equation}
where
$$
d_k(x)=|\{(a_1,\ldots,a_k):\,a_1,\ldots,a_k\in\Z^+, a_1\cdots
a_k=x\}|.
$$
\end{Lem}
Let $K=2^{k-1}$.
\begin{Lem}
\label{welysum} Let $\psi(x)=a_1 x^k+a_2x^{k-1}+\cdots+a_kx$ be a
polynomial with real coefficients and $a_1\in\Z^+$. Then
\begin{align}
\sum_{1\leq x\leq V}e(\alpha\psi(x))\ll
V^{1-\frac{k}{K}}\bigg(V^{k-1}+V^{\frac{k}{2}}(\log
V)^{\frac{k^2-2k}{2}}\bigg(\sum_{y=1}^{V^{k-1}}\min\{V,\|\alpha
k!a_1y\|^{-1})\bigg)^{\frac{1}{2}}\bigg)^{\frac{1}{K}}
\end{align} for
any real $\alpha$. In particular, if $|\alpha-a/q|\leq q^{-2}$ with
$(a,q)=1$, then
\begin{align}
\sum_{1\leq x\leq V}e(\alpha\psi(x))\ll
V\bigg(V^{-\frac{1}{K}}+a_1^\frac{1}{2K}(\log(a_1qV))^{\frac{(k-1)^2}{2K}}\bigg(\frac{1}{q}+\frac{1}{V}+\frac{q}{a_1V^k}\bigg)^{\frac{1}{2K}}\bigg).
\end{align}
\end{Lem}
\begin{proof}
Define the intervals $I_j(V;h_1,\ldots,h_j)$ by
$I_1(V;h_1)=[1,V]\cap[1-h_1,V-h_1]$ and
$$
I_{j+1}(V;h_1,\ldots,h_{j+1})=I_j(V;h_1,\ldots,h_j)\cap\{x:\,x+h_{j+1}\in
I_j(V;h_1,\ldots,h_j)\}.
$$
For $j\geq 1$, we know (cf.\cite{Vaughan97}[Lemma 2.3])that
\begin{align*}
\bigg|\sum_{1\leq x\leq V}e(\alpha \psi(x))\bigg|^{2^j}
\leq&(2V)^{2^j-j-1}\sum_{-V<h_1,\ldots,h_{j}<V}T_{j}(V;h_1,\ldots,h_{j}),
\end{align*}
where
$$
T_j(V;h_1,\ldots,h_j)=\sum_{x\in
I_j(V;h_1,\ldots,h_j)}e(\Delta_j(\alpha \psi(x);h_1,\ldots,h_j)).
$$
In particular,
$$
T_{k-1}(V;h_1,\ldots,h_{k-1})=\sum_{x\in
I_{k-1}(V;h_1,\ldots,h_{k-1})}e(\alpha h_1\cdots h_{k-1}
g_{k-1}(x;h_1,\ldots,h_{k-1})),
$$
where
$$
g_{k-1}(x;h_1,\ldots,h_{k-1})=k!
a_1(x+(h_1+\cdots+h_{k-1})/2)+(k-1)!a_2.
$$
Therefore
\begin{align*}
&\sum_{1\leq x\leq V}e(\alpha\psi(x))\\
\ll&V^{1-\frac{k}{K}}\bigg(\sum_{-V<h_1,\ldots,h_{k-1}<V}\bigg|\sum_{1\leq
x\leq V}e(\alpha
k!a_1h_1\cdots h_{k-1}x)\bigg|\bigg)^{\frac{1}{K}}\\
\ll&V^{1-\frac{k}{K}}\bigg(V^{k-1}+\sum_{\substack{-V<h_1,\ldots,h_{k-1}<V\\
h_1\cdots h_{k-1}\not=0}}\min\{V,\|\alpha
k!a_1h_1\cdots h_{k-1}\|^{-1}\}\bigg)^{\frac{1}{K}}\\
\leq&V^{1-\frac{k}{K}}\bigg(V^{k-1}+\sum_{1\leq y\leq
V^{k-1}}d_{k-1}(y)\min\{V,\|\alpha
k!a_1y\|^{-1}\}\bigg)^{\frac{1}{K}}\\
\leq&V^{1-\frac{k}{K}}\bigg(V^{k-1}+\bigg(\sum_{1\leq y\leq
V^{k-1}}d_{k-1}(y)^2\bigg)^{\frac{1}{2}}\bigg(\sum_{1\leq y\leq
V^{k-1}}\min\{V,\|\alpha
k!a_1y\|^{-1}\}^2\bigg)^{\frac{1}{2}}\bigg)^{\frac{1}{K}}\\
\ll&V^{1-\frac{k}{K}}\bigg(V^{k-1}+V^{\frac{k}{2}}(\log
V)^{\frac{k^2-2k}{2}}\bigg(\sum_{1\leq y\leq
V^{k-1}}\min\{V,\|\alpha k!a_1
y\|^{-1}\}\bigg)^{\frac{1}{2}}\bigg)^{\frac{1}{K}}.
\end{align*}
Finally, if $|\alpha-a/q|\leq q^{-2}$ with $(a,q)=1$, then by
Lemma 2.2 of \cite{Vaughan97}, we have
\begin{align*}
\sum_{1\leq y\leq V^{k-1}}\min\{V,\|\alpha
k!a_1y\|^{-1}\}\leq&\sum_{y=1}^{k!a_1V^{k-1}}\min\{k!a_1V^k/y,\|\alpha
y\|^{-1}\}\\
\ll&
k!a_1V^k\log(2k!a_1V^kq)\bigg(\frac{1}{q}+\frac{1}{V}+\frac{q}{k!a_1V^k}\bigg).
\end{align*}
We are done.
\end{proof}

\begin{Lem}
\label{sumii1a} Suppose that $\psi(x,y)=\sum_{1\leq i,j\leq
k+1}a_{ij}x^{k-i+1}y^{k-j+1}$ is a polynomial with real
coefficients. Suppose that $a_{11}\in\Z^+$ and $a_{12}=0$. Then
\begin{align}&\sum_{1\leq x\leq U}\bigg|\sum_{1\leq y\leq
V}e(\alpha\psi(x,y))\bigg|\notag\\\ll&
UV\bigg(U^{-\frac{1}{K^2}}+V^{-\frac{1}{K^2}}+a_{11}^{\frac{1}{4K^2}}(\log(a_{11}qUV))^{\frac{3k^2-2k+1}{4K^2}}\bigg(\frac{1}{q}+\frac{1}{U}+\frac{q}{a_{11}U^{k}V^k}\bigg)^{\frac{1}{4K^2}}\bigg)
\end{align}
provided that $|\alpha-a/q|\leq q^{-2}$ with $(a,q)=1$.
\end{Lem}
\begin{proof}
Write $\psi(x,y)=\sum_{j=1}^{k+1}\psi_j(x)y^{k-j+1}$. Then by
H\"{o}lder inequality we have
\begin{align*}
&\sum_{1\leq x\leq U}\bigg|\sum_{1\leq y\leq V}e(\alpha \psi(x,y))\bigg|\notag\\
\leq&U^{1-\frac{1}{K}}\bigg(\sum_{1\leq x\leq U}\bigg|\sum_{1\leq y\leq V}e(\alpha \psi(x,y))\bigg|^K\bigg)^{\frac{1}{K}}\notag\\
\ll&U^{1-\frac{1}{K}}V^{1-\frac{k}{K}}\bigg(\sum_{1\leq x\leq
U}\sum_{\substack{|h_1|,\ldots,|h_{k-1}|\leq V\\ y\in
I_{k-1}(V;h_1,\ldots,h_{k-1})}}e(\alpha h_1\cdots h_{k-1}
g_{k-1}(x,y;h_1,\ldots,h_{k-1}))\bigg)^{\frac{1}{K}}\\
\leq&U^{1-\frac{1}{K}}V^{1-\frac{k}{K^2}}\bigg(\sum_{\substack{|h_1|,\ldots,|h_{k-1}|\leq V\\
y\in I_{k-1}(V;h_1,\ldots,h_{k-1})}}\bigg|\sum_{1\leq x\leq
U}e(\alpha h_1\cdots h_{k-1}
g_{k-1}(x,y;h_1,\ldots,h_{k-1}))\bigg|^K\bigg)^{\frac{1}{K^2}},
\end{align*}
where
\begin{align*}
&g_{k-1}(x,y;h_1,\ldots,h_{k-1})\\
=&k! \psi_1(x)(y+(h_1+\cdots+h_{k-1})/2)+(k-1)!\psi_{2}(x).
\end{align*}
Note that $\deg \psi_{2}\leq k-1$ since $a_{12}=0$. Thus applying
Lemma \ref{welysum},
\begin{align*}
&\bigg|\sum_{1\leq x\leq U}e(\alpha h_1\cdots h_{k-1}
g_{k-1}(x,y;h_1,\ldots,h_{k-1}))\bigg|^K\\
\ll&U^{K-k}\bigg(U^{k-1}+U^\frac{k}{2}(\log U)^\frac{k^2-2k}{2}\bigg(\sum_{1\leq z\leq U^{k-1}}\min\{U,\|\alpha (k!)^2a_{11}h_1\cdots h_{k-1}yz\|^{-1}\}\bigg)^{\frac{1}{2}}\bigg)\\
\end{align*}
provided that $h_1\cdots h_{k-1}\not=0$. So
\begin{align*}
&\sum_{\substack{|h_1|,\ldots,|h_{k-1}|\leq V\\
y\in I_{k-1}(V;h_1,\ldots,h_{k-1})}}\bigg|\sum_{1\leq x\leq
U}e(\alpha
h_1\cdots h_{k-1} g_{k-1}(x,y;h_1,\ldots,h_{k-1}))\bigg|^K\\
\ll&U^{K}V^{k-1}+\sum_{\substack{|h_1|,\ldots,|h_{k-1}|\leq
V\\h_1\cdots h_{k-1}\not=0\\y\in
I_{k-1}(V;h_1,\ldots,h_{k-1})}}\bigg|\sum_{1\leq x\leq U}e(\alpha
h_1\cdots h_{k-1} g_{k-1}(x,y;h_1,\ldots,h_{k-1}))\bigg|^K\\
\ll&(\log
U)^\frac{k^2-2k}{2}U^{K-\frac{k}{2}}\sum_{\substack{|h_1|,\ldots,|h_{k-1}|\leq
V\\h_1\cdots h_{k-1}\not=0\\y\in
I_{k-1}(V;h_1,\ldots,h_{k-1})}}\bigg(\sum_{1\leq z\leq U^{k-1}}\min\{U,\|\alpha (k!)^2a_{11}h_1\cdots h_{k-1}yz\|^{-1}\}\bigg)^{\frac{1}{2}}\\
&+U^{K-1}V^{k}+U^{K}V^{k-1}.
\end{align*}
Furthermore,
\begin{align*}
&\sum_{\substack{|h_1|,\ldots,|h_{k-1}|\leq V\\h_1\cdots
h_{k-1}\not=0\\y\in
I_{k-1}(V;h_1,\ldots,h_{k-1})}}\bigg(\sum_{1\leq z\leq U^{k-1}}\min\{U,\|\alpha (k!)^2a_{11}h_1\cdots h_{k-1}yz\|^{-1}\}\bigg)^{\frac{1}{2}}\\
\ll&V^\frac{k}{2}\bigg(\sum_{\substack{|h_1|,\ldots,|h_{k-1}|\leq
V\\h_1\cdots
h_{k-1}\not=0\\ 1\leq y\leq V}}\sum_{1\leq z\leq U^{k-1}}\min\{U,\|\alpha (k!)^2a_{11}h_1\cdots h_{k-1}yz\|^{-1}\}\bigg)^{\frac{1}{2}}\\
\ll&V^\frac{k}{2}\bigg(\sum_{1\leq z\leq U^{k-1}V^k}d_{k+1}(z)\min\{U,\|\alpha (k!)^2a_{11}z\|^{-1}\}\bigg)^{\frac{1}{2}}\\
\ll&V^\frac{k}{2}\bigg(\sum_{1\leq z\leq
U^{k-1}V^k}d_{k+1}(z)^2\bigg)^{\frac{1}{4}}
\bigg(\sum_{1\leq z\leq U^{k-1}V^k}\min\{U,\|\alpha (k!)^2a_{11}z\|^{-1}\}^2\bigg)^{\frac{1}{4}}\\
\ll&V^\frac{3k}{4}U^{\frac{k}{4}}(\log(UV))^{\frac{k^2+2k}{4}}
\bigg(\sum_{1\leq z\leq (k!)^2a_{11}U^{k-1}V^k}\min\{(k!)^2a_{11}U^{k}V^k/z,\|\alpha z\|^{-1}\}\bigg)^{\frac{1}{4}}\\
\ll&V^\frac{3k}{4}U^{\frac{k}{4}}(\log(UV))^{\frac{k^2+2k}{4}}\cdot
a_{11}^{\frac{1}{4}}U^\frac{k}{4}V^\frac{k}{4}(\log(a_{11}qUV))^{\frac{1}{4}}
\bigg(\frac{1}{q}+\frac{1}{U}+\frac{q}{(k!)^2a_{11}U^{k}V^k}\bigg)^{\frac{1}{4}}.
\end{align*}
All are done.
\end{proof}

\begin{Lem}
\label{sumii2a} Suppose that $\psi(x,y)=\sum_{1\leq i,j\leq
k+1}a_{ij}x^{k-i+1}y^{k-j+1}$ is a polynomial with real
coefficients. Suppose that $a_{11}\in\Z^+$ and $a_{12}=0$. Then
\begin{align} &\sum_{U\leq x\leq 2U}\bigg|\sum_{1\leq y\leq
V/x}e(\alpha\psi(x,y))\bigg|\notag\\
\ll&V\bigg(U^{\frac{1}{K}}V^{-\frac{1}{K}}+U^{-\frac{1}{K^2}}+
+a_{11}^\frac{1}{4K^2}(\log(a_{11}qUV))^{\frac{3k^2-2k+1}{4K^2}}\bigg(\frac{1}{q}+\frac{1}{U}+\frac{q}{a_{11}V^k}\bigg)^{\frac{1}{4K^2}}\bigg)
\end{align}
provided that $|\alpha-a/q|\leq q^{-2}$ with $(a,q)=1$.
\end{Lem}
\begin{proof}
Write $\psi(x,y)=\sum_{j=1}^{k+1}\psi_j(x)y^{k-j+1}$. And let
$$
T_{k-1}(x,Q;h_1,\ldots,h_{k-1})=\sum_{y\in
I_{k-1}(Q;h_1,\ldots,h_{k-1})}e(\alpha h_1\cdots h_{k-1}
g_{k-1}(x,y;h_1,\ldots,h_{k-1}))
$$
where
\begin{align*}
&g_{k-1}(x,y;h_1,\ldots,h_{k-1})\\
=&k! \psi_1(x)(y+(h_1+\cdots+h_{k-1})/2)+(k-1)!\psi_2(x).
\end{align*}
Then
\begin{align*}
&\sum_{U\leq x\leq 2U}\bigg|\sum_{1\leq y\leq V/x}e(\alpha \psi(x,y))\bigg|\notag\\
\leq&U^{1-\frac{1}{K}}\bigg(\sum_{U\leq x\leq
2U}(V/x)^{K-k}\sum_{|h_1|,\ldots,|h_{k-1}|\leq V/x}
T_{k-1}(x,\floor{V/x};h_1,\ldots,h_{k-1})\bigg)^{\frac{1}{K}}\\
\leq&U^{1-\frac{1}{K}}V^{1-\frac{k}{K}}\bigg(\sum_{|h_1|,\ldots,|h_{k-1}|\leq
V/U}\sum_{\substack{U\leq x\leq 2U\\ x\leq\min_{1\leq
i<k}V/|h_i|}}x^{k-K}T_{k-1}(x,\floor{V/x};h_1,\ldots,h_{k-1})\bigg)^{\frac{1}{K}}.
\end{align*}

By an induction on $j$, it is not difficult to prove that
$I_j(Q_1;h_1,\ldots,h_j)\subseteq I_j(Q_2;h_1,\ldots,h_j)$ if
$Q_1\leq Q_2$. Hence for any $y$, the set
$$
\bar{I}(y;h_1,\ldots,h_j)=\{x:\,y\in
I_j(\floor{V/x};h_1,\ldots,h_j)\}
$$
is exactly an interval. Then
\begin{align*}
&\sum_{\substack{U\leq x\leq
2U\\ 1\leq x\leq\min_{1\leq i<k}V/|h_i|}}x^{k-K}T_{k-1}(x;h_1,\ldots,h_{k-1})\\
=&\sum_{y\in
I_{k-1}(V/U;h_1,\ldots,h_{k-1})}\sum_{\substack{U\leq x\leq 2U\\
1\leq x\leq\min_{1\leq i<k}V/|h_i|\\
x\in\bar{I}(y;h_1,\ldots,h_{k-1})}}x^{k-K}e(\alpha h_1\cdots
h_{k-1} g_{k-1}(x,y;h_1,\ldots,h_{k-1})).
\end{align*}
By Lemma \ref{welysum} we know that
\begin{align*}
&\sum_{x=Q_1}^{Q_2}x^{k-K}e(\alpha h_1\cdots h_{k-1}
g_{k-1}(x,y;h_1,\ldots,h_{k-1}))\\
=&Q_2^{k-K}\sum_{x=1}^{Q_2}e(\alpha h_1\cdots h_{k-1}
g_{k-1}(x,y;h_1,\ldots,h_{k-1}))\\
&-Q_1^{k-K}\sum_{x=1}^{Q_1-1}e(\alpha h_1\cdots h_{k-1}
g_{k-1}(x,y;h_1,\ldots,h_{k-1}))
\\
&-\sum_{X=Q_1}^{Q_2-1}((X+1)^{k-K}-X^{k-K})\sum_{x=1}^Xe(\alpha
h_1\cdots h_{k-1}
g_{k-1}(x,y;h_1,\ldots,h_{k-1}))\\
\ll&Q_1^{k-K-1}Q_2^{2-\frac{k}{K}}\bigg(Q_2^{k-1}+Q_2^{\frac{k}{2}}(\log
Q_2)^\frac{k^2-2k}{2}\\
&\cdot\bigg(\sum_{z=1}^{Q_2^{k-1}}\min\{Q_2,\|\alpha
(k!)^2a_{11}h_1\cdots h_{k-1}y
z\|^{-1}\}\bigg)^{\frac{1}{2}}\bigg)^{\frac{1}{K}}.
\end{align*}
Therefore
\begin{align*}
&\sum_{|h_1|,\ldots,|h_{k-1}|\leq V/U}\sum_{\substack{U\leq x\leq 2U\\
1\leq x\leq\min_{1\leq i<k}V/|h_i|}}x^{k-K}T_{k-1}(x;h_1,\ldots,h_{k-1})\\
\ll&V^{k-1}U^{2-K}+U^{k-K+1-\frac{k}{K}}
\sum_{\substack{|h_1|,\ldots,|h_{k-1}|\leq V/U\\
h_1\cdots h_{k-1}\not=0}}\sum_{1\leq y\leq V/U}
\bigg(U^{k-1}\\
&+U^{\frac{k}{2}}(\log U)^\frac{k^2-2k}{2}\bigg(\sum_{1\leq
z\leq(2U)^{k-1}}\min\{2U,\|\alpha (k!)^2a_{11}h_1\cdots
h_{k-1}y z\|^{-1}\}\bigg)^{\frac{1}{2}}\bigg)^{\frac{1}{K}}\\
\leq&V^{k-1}U^{2-K}+U^{1-K} V^{k-\frac{k}{K}}
\bigg(V^kU^{-1}+U^{\frac{k}{2}}(\log
U)^\frac{k^2-2k}{2}\\
&\cdot\sum_{\substack{|h_1|,\ldots,|h_{k-1}|\leq V/U\\
h_1\cdots h_{k-1}\not=0}}\sum_{1\leq y\leq V/U}
\bigg(\sum_{1\leq
z\leq(2U)^{k-1}}\min\{2U,\|\alpha (k!)^2a_{11}h_1\cdots h_{k-1}y
z\|^{-1}\}\bigg)^{\frac{1}{2}}\bigg)^{\frac{1}{K}}.
\end{align*}
Also,
\begin{align*}
&\sum_{\substack{|h_1|,\ldots,|h_{k-1}|\leq V/U\\
h_1\cdots h_{k-1}\not=0}}\sum_{1\leq y\leq V/U}\bigg(\sum_{1\leq
z\leq(2U)^{k-1}}\min\{2U,\|\alpha
(k!)^2a_{11}h_1\cdots h_{k-1}y z\|^{-1}\}\bigg)^{\frac{1}{2}}\\
\ll&V^{\frac{k}{2}}U^{-\frac{k}{2}}\bigg(\sum_{\substack{|h_1|,\ldots,|h_{k-1}|\leq V/U\\
h_1\cdots h_{k-1}\not=0}}\sum_{1\leq y\leq V/U}\sum_{1\leq
z\leq(2U)^{k-1}}\min\{2U,\|\alpha
(k!)^2a_{11}h_1\cdots h_{k-1}y z\|^{-1}\}\bigg)^{\frac{1}{2}}\\
\leq&V^{\frac{k}{2}}U^{-\frac{k}{2}}\bigg(\sum_{z=1}^{2^{k-1}V^kU^{-1}}d_{k+1}(z)\min\{2U,\|\alpha
(k!)^2a_{11}z\|^{-1}\}\bigg)^{\frac{1}{2}}\\
\ll&V^{\frac{k}{2}}U^{-\frac{k}{2}}\cdot
V^\frac{k}{4}(\log(V^kU^{-1}))^{\frac{k^2+2k}{4}}
\bigg(\sum_{1\leq z\leq2^{k-1}V^kU^{-1}}\min\{2U,\|\alpha
(k!)^2a_{11}z\|^{-1}\}\bigg)^{\frac{1}{4}}\\
\ll&V^{\frac{3k}{4}}U^{-\frac{k}{2}}(\log(VU^{-1}))^{\frac{k^2+2k}{4}}
\cdot
a_{11}^\frac{1}{4}V^\frac{k}{4}(\log(a_{11}qV))^{\frac{1}{4}}
\bigg(\frac{1}{q}+\frac{1}{U}+\frac{q}{a_{11}V^k}\bigg)^{\frac{1}{4}}.
\end{align*}
It follows that
\begin{align*}
&\sum_{U\leq x\leq 2U}\bigg|\sum_{1\leq y\leq V/x}e(\alpha \psi(x,y))\bigg|\notag\\
\ll&U^{1-\frac{1}{K}}V^{1-\frac{k}{K}}\bigg(V^{k-1}U^{2-K}+U^{1-K}
V^{k-\frac{k}{K}} \bigg(V^kU^{-1}\\
&+a_{11}^\frac{1}{4}V^{k}(\log(a_{11}qUV))^{\frac{3k^2-2k+1}{4}}\bigg(\frac{1}{q}+\frac{1}{U}+\frac{q}{a_{11}V^k}\bigg)^{\frac{1}{4}}\bigg)^{\frac{1}{K}}\bigg)^{\frac{1}{K}}\\
\ll&V\bigg(U^{\frac{1}{K}}V^{-\frac{1}{K}}+U^{-\frac{1}{K^2}}+
+a_{11}^\frac{1}{4K^2}(\log(a_{11}qUV))^{\frac{3k^2-2k+1}{4K^2}}\bigg(\frac{1}{q}+\frac{1}{U}+\frac{q}{a_{11}V^k}\bigg)^{\frac{1}{4K^2}}\bigg).
\end{align*}

\end{proof}

From Lemma \ref{sumii2a}, it is easily derived that
\begin{align} &\sum_{U_1\leq x\leq U_2}\sum_{y\leq
V/x}e(\alpha\psi(x,y))\notag\\
\ll&V\log
U_2\bigg(U_2^{\frac{1}{K}}V^{-\frac{1}{K}}+U_1^{-\frac{1}{K^2}}+
a_{11}^\frac{1}{4K^2}(\log(a_{11}qU_2V))^{\frac{3k^2-2k+1}{4K^2}}\bigg(\frac{1}{q}+\frac{1}{U_1}+\frac{q}{a_{11}V^k}\bigg)^{\frac{1}{4K^2}}\bigg).
\end{align}

\begin{Lem}
\label{sumii1b} Suppose that $\psi(x,y)=\sum_{1\leq i,j\leq
k+1}a_{ij}x^{k-i+1}y^{k-j+1}$ is a polynomial with real
coefficients and $a_{11},a_{21},\ldots,a_{k+1,1}\in\Z$. If
$a_{11}x^k+a_{21}x^{k-1}+\cdots+a_{k+1,1}\not=0$ for each $1\leq
x\leq U$, then
\begin{align} &\sum_{1\leq x\leq U}\bigg|\sum_{1\leq y\leq
V}e(\alpha\psi(x,y))\bigg|\notag\\
\ll&_\epsilon
UV\bigg(V^{-\frac{1}{K}}U^{\frac{1}{K}}+a_{*}^\frac{1}{2K}U^\frac{k}{2K}(\log(a_*qU^kV))^{\frac{(k-1)^2}{2K}}\bigg(\frac{1}{q}+\frac{U}{V}+\frac{q}{V^k}\bigg)^{\frac{1}{2K}}\bigg)
\end{align}
provided that $|\alpha-a/q|\leq q^{-2}$ with $(a,q)=1$, where
$a_*=|a_{11}|+|a_{21}|+\cdots+|a_{k+1,1}|$.
\end{Lem}
\begin{proof}
Write $\psi(x,y)=\sum_{j=1}^{k+1}\psi_j(x)y^{k-j+1}$. Then by
Lemma \ref{welysum} we have
\begin{align*} &\sum_{1\leq x\leq U}\bigg|\sum_{1\leq y\leq
V}e(\alpha\psi(x,y))\bigg|\\
\ll&\sum_{1\leq x\leq U}
V\bigg(V^{-\frac{1}{K}}+|\psi_1(x)|^\frac{1}{2K}(\log(|\psi_1(x)|qV))^{\frac{(k-1)^2}{2K}}\bigg(\frac{1}{q}+\frac{1}{V}+\frac{q}{|\psi_1(x)|V^k}\bigg)^{\frac{1}{2K}}\bigg)\\
\ll&
UV\bigg(V^{-\frac{1}{K}}+a_{*}^\frac{1}{2K}U^\frac{k}{2K}(\log(a_*qU^kV))^{\frac{(k-1)^2}{2K}}\bigg(\frac{1}{q}+\frac{1}{V}+\frac{q}{V^k}\bigg)^{\frac{1}{2K}}\bigg).
\end{align*}
\end{proof}

\begin{Lem}
\label{sumii2b} Suppose that $\psi(x,y)=\sum_{1\leq i,j\leq
k+1}a_{ij}x^{k-i+1}y^{k-j+1}$ is a polynomial with real
coefficients and $a_{11},a_{21},\ldots,a_{k+1,1}\in\Z$. If
$a_{11}x^k+a_{21}x^{k-1}+\cdots+a_{k+1,1}\not=0$ for each $1\leq
x\leq U$, then
\begin{align} &\sum_{1\leq x\leq U}\bigg|\sum_{1\leq y\leq
V/x}e(\alpha\psi(x,y))\bigg|\notag\\
\ll&V(\log
U)\bigg(V^{-\frac{1}{K}}U^{\frac{1}{K}}+a_*^\frac{1}{2K}U^\frac{k}{2K}(\log(a_*qU^kV))^{\frac{(k-1)^2}{2K}}\bigg(\frac{1}{q}+\frac{U}{V}+\frac{q}{V^k}\bigg)^{\frac{1}{2K}}\bigg)
\end{align}
provided that $|\alpha-a/q|\leq q^{-2}$ with $(a,q)=1$, where
$a_*=|a_{11}|+|a_{21}|+\cdots+|a_{k+1,1}|$.
\end{Lem}
\begin{proof}
Write $\psi(x,y)=\sum_{j=1}^{k+1}\psi_j(x)y^{k-j+1}$. Now
$\psi_1(x)$ is a polynomial with integral coefficients. Hence by
Lemma \ref{welysum} we have
\begin{align*} &\sum_{1\leq x\leq U}\bigg|\sum_{1\leq y\leq
V/x}e(\alpha\psi(x,y))\bigg|\\
\ll&\sum_{1\leq x\leq U}
Vx^{-1}\bigg(V^{-\frac{1}{K}}x^{\frac{1}{K}}+|\psi_1(x)|^\frac{1}{2K}(\log(|\psi_1(x)|qV))^{\frac{(k-1)^2}{2K}}\bigg(\frac{1}{q}+\frac{x}{V}+\frac{qx^k}{|\psi_1(x)|V^k}\bigg)^{\frac{1}{2K}}\bigg)\\
\ll& V(\log
U)\bigg(V^{-\frac{1}{K}}U^{\frac{1}{K}}+a_{*}^\frac{1}{2K}U^\frac{k}{2K}(\log(a_*qU^kV))^{\frac{(k-1)^2}{2K}}\bigg(\frac{1}{q}+\frac{U}{V}+\frac{q}{V^k}\bigg)^{\frac{1}{2K}}\bigg).
\end{align*}

\end{proof}

\begin{Lem}
\label{sumii2w} Let $\psi(x)=a_1x^k+a_2x^{k-1}\cdots+a_kx$ be a
polynomial with integral coefficients and $a_1\in\Z^+$. Let
$A\geq1$ and $B>32k^2(k^2+K^2)A$. Suppose that $1\leq W,
a_1\leq(\log V)^A$ and $1\leq U\leq V^{1-\delta}$ for some
$\delta>0$. Then for any integer $b$ and $1\leq c,c'\leq W$ with
$cc'\equiv b\pmod{W}$, we have
\begin{align}
\sum_{\substack{1\leq x\leq U\\
x\equiv c\pmod{W}}}\bigg|\sum_{\substack{1\leq y\leq V/x\\ y\equiv
c'\pmod{W}}}e(\alpha \psi((xy-b)/W))\bigg|\ll_{A,B} V(\log
V)^{-\frac{B}{16k^2K^2}}
\end{align}
provided that $|\alpha-a/q|\leq q^{-2}$ with $(\log V)^B\leq q\leq
\psi(V)(\log V)^{-B}$ and $(a,q)=1$.
\end{Lem}
\begin{proof}
Let
$$
U_*=\min\{q^{\frac{1}{2k^2}},(2V^k/W^kq)^{\frac{1}{2k^2}},U\}.
$$
Apparently $U_*\geq \min\{(\log V)^{\frac{B-(k+1)A}{2k^2}},U\}$
and
$$
U_*\leq (q\cdot 2V^k/W^kq)^{\frac{1}{4k^2}}\ll V^{\frac{1}{4k}}.
$$
Rewrite
\begin{align*}
&\sum_{\substack{1\leq x\leq U\\
x\equiv c\pmod{W}}}\bigg|\sum_{\substack{1\leq y\leq V/x\\y\equiv
c'\pmod{W}}}e(\alpha
\psi((xy-b)/W))\bigg|\\
=&\bigg(\sum_{\substack{1\leq x\leq U_*\\
x\equiv c\pmod{W}}}+\sum_{\substack{U_*<x\leq U\\
x\equiv c\pmod{W}}}\bigg)\bigg|\sum_{\substack{1\leq y\leq
V/x\\y\equiv c'\pmod{W}}}e(\alpha \psi((xy-b)/W))\bigg|.
\end{align*}

Clearly
\begin{align*}
&\sum_{\substack{1\leq x\leq U_*\\
x\equiv c\pmod{W}}}\bigg|\sum_{\substack{1\leq y\leq V/x\\y\equiv
c'\pmod{W}}}e(\alpha
\psi((xy-b)/W))\bigg|\\
=&\frac{1}{W}\sum_{j=1}^W\sum_{\substack{1\leq x\leq
U_*}}e(j(x-c)/W)\bigg|\sum_{\substack{1\leq y\leq V/Wx}}e(\alpha
\psi((x(Wy+c')-b)/W))+O(1)\bigg|\\
\leq&\sum_{\substack{1\leq x\leq U_*}}\bigg|\sum_{\substack{1\leq
y\leq (V/W)/x}}e(\alpha \psi(xy+(xc'-b)/W))\bigg|+O(U_*).
\end{align*}
By Lemma \ref{sumii2b},
\begin{align*}
&\sum_{\substack{1\leq x\leq U_*}}\bigg|\sum_{\substack{1\leq y\leq
(V/W)/x}}e(\alpha
\psi(xy+(xc'-b)/W))\bigg|\\
\ll&V\log
U_*\bigg(V^{-\frac{1}{K}}W^{\frac{1}{K}}U_*^{\frac{1}{K}}+a_{1}^\frac{1}{2K}U_*^\frac{k}{2K}
(\log(a_1qVW^{-1}U_*))^{\frac{(k-1)^2}{2K}}\bigg(\frac{1}{q}+\frac{U_*}{V}+\frac{qW^k}{V^k}\bigg)^{\frac{1}{2K}}\bigg).
\end{align*}
If $U_*=U$, then $U\leq (\log V)^{\frac{B-(k+1)A}{2k^2}}$ and
\begin{align*}
&V\log U_*\bigg(V^{-\frac{1}{K}}W^{\frac{1}{K}}U_*^{\frac{1}{K}}+a_{1}^\frac{1}{2K}U_*^\frac{k}{2K}(\log(a_1qVW^{-1}U_*))^{\frac{(k-1)^2}{2K}}\bigg(\frac{1}{q}+\frac{U_*}{V}+\frac{qW^k}{V^k}\bigg)^{\frac{1}{2K}}\bigg)\\
\ll&V\log U(V^{-\frac{1}{K}}W^{\frac{1}{K}}U^{\frac{1}{K}}+a_{1}^\frac{1}{2K}(\log V)^{\frac{(k-1)^2}{2K}}(U^{\frac{k}{2K}}q^{-\frac1{2K}}+U^\frac{k+1}{2K}V^{-\frac{1}{2K}}+(\log V)^\frac{(2k^2+k)A-(2k-1)B}{4kK}))\\
\ll&_{A,B}V(\log
V)^{1+\frac{(k-1)^2+(k+2)A}{2K}-\frac{(2k-1)B}{4kK}}.
\end{align*}
Below we assume that $U_*<U$, then $(\log
V)^{\frac{B-(k+1)A}{2k^2}}\ll U_* \ll V^{\frac{1}{4k}}$ and
\begin{align*}
&V\log U_*\bigg(V^{-\frac{1}{K}}W^{\frac{1}{K}}U_*^{\frac{1}{K}}+a_{1}^\frac{1}{2K}U_*^\frac{k}{2K}(\log(a_1qVW^{-1}U_*))^{\frac{(k-1)^2}{2K}}\bigg(\frac{1}{q}+\frac{U_*}{V}+\frac{qW^k}{V^k}\bigg)^{\frac{1}{2K}}\bigg)\\
\ll&V\log V(V^{-\frac{1}{K}}W^{\frac{1}{K}}U_*^{\frac{1}{K}}+a_{1}^\frac{1}{2K}(\log V)^{\frac{(k-1)^2}{2K}}(U_*^{\frac{k}{2K}}q^{-\frac1{2K}}+U_*^\frac{k+1}{2K}V^{-\frac{1}{2K}}+U_*^\frac{k-2k^2}{2K}))\\
\ll&_{A,B}V(\log
V)^{1+\frac{(k-1)^2+(k+2)A}{2K}-\frac{(2k-1)B}{4kK}}.
\end{align*}

On the other hand,
\begin{align*}
&\sum_{\substack{U_*<x\leq U\\
x\equiv c\pmod{W}}}\bigg|\sum_{\substack{1\leq y\leq V/x\\y\equiv
c'\pmod{W}}}e(\alpha
\psi((xy-b)/W))\bigg|\\
=&\sum_{\substack{(U_*-c)/W<x\leq
(U-c)/W}}\bigg|\sum_{\substack{1\leq y\leq V/(Wx+c)\\ y\equiv
c'\pmod{W}}}e(\alpha
\psi(xy+(yc-b)/W))\bigg|\\
=&\sum_{\substack{(U_*-c)/W<x\leq
(U-c)/W}}\bigg|\sum_{\substack{1\leq y\leq V/Wx\\ y\equiv
c'\pmod{W}}}e(\alpha
\psi(xy+(yc-b)/W))+O(V/W^2x^2)\bigg|\\
=&\sum_{\substack{(U_*-c)/W<x\leq
(U-c)/W}}\bigg|\sum_{\substack{1\leq y\leq V/Wx\\ y\equiv
c'\pmod{W}}}e(\alpha \psi(xy+(yc-b)/W))\bigg|+O(V/WU_*).
\end{align*}
Notice that
\begin{align*}
&\sum_{\substack{(U_*-c)/W<x\leq
(U-c)/W}}\bigg|\sum_{\substack{1\leq y\leq V/Wx\\ y\equiv
c'\pmod{W}}}e(\alpha
\psi(xy+(yc-b)/W))\bigg|\\
=&\frac{1}{W}\sum_{\substack{(U_*-c)/W<x\leq
(U-c)/W}}\bigg|\sum_{\substack{1\leq y\leq V/Wx}}e(\alpha
\psi(xy+(yc-b)/W))\sum_{j=1}^We((y-c')j/W)\bigg|\\
\leq&\max_{1\leq j\leq W}\sum_{\substack{(U_*-c)/W<x\leq
(U-c)/W}}\bigg|\sum_{\substack{1\leq y\leq V/Wx}}e(\alpha
\psi(xy+(yc-b)/W+(y-c')j/W)\bigg|.
\end{align*}
Hence by Lemma \ref{sumii2a},
\begin{align*}
&\sum_{\substack{(U_*-c)/W<x\leq
(U-c)/W}}\bigg|\sum_{\substack{1\leq y\leq (V/W)/x}}e(\alpha
\psi(xy+(yc-b)/W+(y-c')j/W)\bigg|\\
\ll&VW^{-1}\log(UW^{-1})
\bigg(U^{\frac{1}{K}}V^{-\frac{1}{K}}+U_*^{-\frac{1}{K^2}}W^{\frac{1}{K^2}}+
a_{1}^\frac{1}{4K^2}(\log V)^{\frac{3k^2-2k+1}{4K^2}}U_*^{-\frac{1}{4K^2}}W^{\frac{1}{4K^2}}\bigg)\\
\ll&_{A,B}V(\log
V)^{1+\frac{3k^2-2k+1+3A}{4K^2}-\frac{B}{8k^2K^2}}.
\end{align*}

\end{proof}

\begin{Lem}
\label{sumii1w} Let $\psi(x)=a_1x^k+a_2x^{k_1}\cdots+a_kx$ be a
polynomial with integral coefficients and $a_1\in\Z^+$. Let
$A\geq1$ and $B>16k^3A$. Suppose that $1\leq W,
a_1\leq(\log(UV))^A$. Let $g(x)$ be a polynomial with the degree
at most $k$ satisfying that the coefficient of $x^k$ in $g(Wx)$ is
an integer. Then for any integer $b$ and $1\leq c,c'\leq W$ with
$cc'\equiv b\pmod{W}$, we have
\begin{align}
\sum_{\substack{1\leq x\leq U\\
x\equiv c\pmod{W}}}\bigg|\sum_{\substack{1\leq y\leq V\\y\equiv
c'\pmod{W}}}e(\alpha(\psi((xy-b)/W)+g(y)))\bigg|\ll_{A,B}
UV(\log(UV))^{-\frac{B}{16kK^2}}
\end{align}
provided that $|\alpha-a/q|\leq q^{-2}$ with $(\log V)^B\leq q\leq
\psi(V)(\log V)^{-B}$ and $(a,q)=1$.
\end{Lem}
\begin{proof}
Suppose that $U\geq (\log V)^{\frac{B}{2k}}$. Then by Lemma
\ref{sumii1a},
\begin{align*}
&\sum_{\substack{1\leq x\leq U\\
x\equiv c\pmod{W}}}\bigg|\sum_{\substack{1\leq y\leq V\\y\equiv
c'\pmod{W}}}e(\alpha(\psi((xy-b)/W)+g(y)))\bigg|\\
=&\sum_{\substack{1\leq x\leq
(U-c)/W+1}}\frac{1}{W}\bigg|\sum_{j=1}^W\sum_{\substack{1\leq
y\leq V}}e(\alpha(\psi(xy-(W-c)y/W-b/W)+g(y))+j(y-c')/W)\bigg|\\
\leq&\frac{1}{W}\sum_{j=1}^W\sum_{\substack{1\leq x\leq
(U-c)/W+1}}\bigg|\sum_{\substack{1\leq y\leq V}}e(\alpha
(\psi(xy-(W-c)y/W-b/W)+g(y))+j(y-c')/W)\bigg|\\
\ll&
UV\bigg(U^{-\frac{1}{K^2}}W^{\frac{1}{K^2}}+V^{-\frac{1}{K^2}}+a_{11}^{\frac{1}{4K^2}}(\log(a_{11}qUV))^{\frac{3k^2-2k+1}{4K^2}}\bigg(\frac{1}{q}+\frac{W}{U}+\frac{qW^k}{a_{11}U^{k}V^k}\bigg)^{\frac{1}{4K^2}}\bigg)\\
\ll&_{A,B}UV(\log(UV))^{\frac{3k^2-2k+1+(k+1)A}{4K^2}-\frac{B}{8kK^2}}.
\end{align*}
Also, if $U\leq (\log V)^{\frac{B}{2k}}$, then by Lemma
\ref{sumii1b},
\begin{align*}
&\sum_{\substack{1\leq x\leq U\\
x\equiv c\pmod{W}}}\bigg|\sum_{\substack{1\leq y\leq V\\y\equiv
c'\pmod{W}}}e(\alpha(\psi((xy-b)/W)+g(y)))\bigg|\\
\leq&\sum_{\substack{1\leq x\leq U}}\bigg|\sum_{\substack{1\leq
y\leq (V-c')/W+1}}e(\alpha (\psi(xy-(W-c')x/W-b/W)+g(Wy-W+c')))\bigg|\\
\ll&
UV\bigg(V^{-\frac{1}{K}}W^{\frac{1}{K}}+a_{*}^\frac{1}{2K}U^\frac{k}{2K}(\log(a_*qU^kV))^{\frac{(k-1)^2}{2K}}\bigg(\frac{1}{q}+\frac{W}{V}+\frac{qW^k}{V^k}\bigg)^{\frac{1}{2K}}\bigg)\\
\ll&_{A,B}UV(\log(UV))^{\frac{(k-1)^2+(k+2)A}{2K}-\frac{B}{4K}}.
\end{align*}
\end{proof}

\begin{Thm} Let $\psi(x)=a_1x^k+a_2x^{k-1}\cdots+a_kx$ be a
polynomial with integral coefficients and $a_1\in\Z^+$. Let
$A\geq1$ and $B>64k^2(k^2+K^2)A$. Suppose that $1\leq W,
a_1\leq(\log N)^A$. Then we have
\begin{align}
\sum_{\substack{1\leq x\leq N\\ Wx+b\, {\rm is\, prime}}}\log
(Wx+b)e(\alpha \psi(x))\ll_{A,B} N(\log N)^{-\frac{B}{64k^2K^2}}
\end{align}
provided that $|\alpha-a/q|\leq q^{-2}$ with $(\log N)^{B+1}\leq
q\leq \psi(N)(\log N)^{-B-1}$ and $(a,q)=1$.
\end{Thm}
\begin{proof}
For a proposition $P$, define $\1_P=1$ or $0$ according to whether
$P$ holds. Let $F(x)=e(\alpha \psi((x-b)/W))\1_{x\equiv b\pmod{W}}$.
Let $V=WN+b$ and $X=V^{2/5}$. Clearly
$$
(\log V)^{B}\leq (\log N)^{B+1}\leq q\leq\psi(N)(\log
N)^{-B-1}\leq \psi(V)(\log V)^{-B}.
$$
By Vaughan's identity we have,
$$
\sum_{X<x\leq V}\Lambda(x)F(x)=S_1-S_2-S_3,
$$
where
$$
S_1=\sum_{1\leq d\leq X}\mu(d)\sum_{1\leq z\leq V/d}\sum_{x\leq
V/dz}\Lambda(x)F(xdz),
$$
$$
S_2=\sum_{1\leq d\leq X}\mu(d)\sum_{1\leq z\leq
V/d}\sum_{x\leq\min\{X, V/dz\}}\Lambda(x)F(xdz),
$$
and
$$
S_3=\sum_{X<u\leq V}\sum_{\substack{1\leq d\leq X\\ d\mid
u}}\mu(d)\sum_{X<x\leq V/u}\Lambda(x)F(xu).
$$
In fact, letting $\tau_u=\sum_{1\leq d\mid u, d\leq X}\mu(d)$, we
have
$$
\sum_{1\leq u\leq V}\tau_u\sum_{X<x\leq
V/u}\Lambda(x)F(xu)=\sum_{X<u\leq V}\tau_u\sum_{X<x\leq
V/u}\Lambda(x)F(xu)+\sum_{X<x\leq V}\Lambda(x)F(x),
$$
since $\tau_1=1$ and $\tau_u=0$ for $1<u\leq X$. On the other hand,
\begin{align*}
\sum_{1\leq u\leq V}\tau_u\sum_{X<x\leq V/u}\Lambda(x)F(xu)=&
\sum_{1\leq u\leq V}\sum_{d\mid u, 1\leq d\leq X}\mu(d)\sum_{X<x\leq
V/u}\Lambda(x)F(xu)\\
=&\sum_{1\leq d\leq X}\mu(d)\sum_{1\leq z\leq V/d}\sum_{X<x\leq
V/dz}\Lambda(x)F(xdz).
\end{align*}

First, we compute
\begin{align*}
|S_1|=&\bigg|\sum_{d\leq X}\mu(d)\sum_{xz\leq V/d}\Lambda(x)e(\alpha
\psi((dxz-b)/W))\1_{dxz\equiv
b\pmod{W}}\bigg|\\
=&\bigg|\sum_{1\leq d\leq X}\mu(d)\sum_{1\leq u\leq V/d}e(\alpha
\psi((du-b)/W))\1_{du\equiv
b\pmod{W}}\sum_{x\mid u}\Lambda(x)\bigg|\\
\leq&\sum_{1\leq d\leq X}\bigg|\sum_{1\leq u\leq V/d}e(\alpha
\psi((du-b)/W))\1_{du\equiv
b\pmod{W}}\log u\bigg|\\
\leq&\sum_{1\leq d\leq X}\bigg|\sum_{1\leq u\leq V/d}e(\alpha
\psi((du-b)/W))\1_{du\equiv
b\pmod{W}}\int_1^u\frac{dt}{t}\bigg|\\
\leq&\int_1^{V}\sum_{1\leq d\leq X}\bigg|\sum_{t\leq u\leq
V/d}e(\alpha \psi((du-b)/W))\1_{du\equiv
b\pmod{W}}\bigg|\frac{dt}{t}.
\end{align*}
Clearly
\begin{align*}
&\sum_{1\leq d\leq X}\bigg|\sum_{t\leq u\leq V/d}e(\alpha
\psi((du-b)/W))\1_{du\equiv b\pmod{W}}\bigg|\\
=&\sum_{1\leq d\leq \min\{X,V/t\}}\bigg|\sum_{t\leq u\leq
V/d}e(\alpha
\psi((du-b)/W))\1_{du\equiv b\pmod{W}}\bigg|\\
=&\sum_{\substack{1\leq c\leq W\\(c,W)=1}}\sum_{\substack{1\leq d\leq\min\{X,V/t\}\\
d\equiv c\pmod{W}}}\bigg|\sum_{\substack{t\leq u\leq V/d\\ uc\equiv
b\pmod{W}}}e(\alpha \psi((du-b)/W))\bigg|.
\end{align*}
So it suffices to estimate
\begin{align*}
\sum_{\substack{1\leq d\leq\min\{X,V/t\}\\
d\equiv c\pmod{W}}}\bigg|\sum_{\substack{1\leq u<t\\ uc\equiv
b\pmod{W}}}e(\alpha \psi((du-b)/W))\bigg|
\end{align*}
and
\begin{align*}
\sum_{\substack{1\leq d\leq\min\{X,V/t\}\\
d\equiv c\pmod{W}}}\bigg|\sum_{\substack{1\leq u\leq V/d\\ uc\equiv
b\pmod{W}}}e(\alpha \psi((du-b)/W))\bigg|.
\end{align*}
Applying Lemma \ref{sumii2w},
\begin{align*}
\sum_{\substack{1\leq d\leq\min\{X,V/t\}\\
d\equiv c\pmod{W}}}\bigg|\sum_{\substack{1\leq u\leq V/d\\ uc\equiv
b\pmod{W}}}e(\alpha \psi((du-b)/W))\bigg|\ll V(\log
V)^{-\frac{B}{16k^2K^2}}.
\end{align*}
Since
\begin{align*}
\sum_{\substack{1\leq d\leq\min\{X,V/t\}\\
d\equiv c\pmod{W}}}\bigg|\sum_{\substack{1\leq u<t\\ uc\equiv
b\pmod{W}}}e(\alpha \psi((du-b)/W))\bigg|\leq Xt,
\end{align*}
we may assume that $t\geq V^{\frac{1}{2}}$. Then by Lemma
\ref{sumii1w},
\begin{align*}
&\sum_{\substack{1\leq d\leq \min\{X,V/t\}\\
d\equiv c\pmod{W}}}\bigg|\sum_{\substack{1\leq u<t\\ uc\equiv
b\pmod{W}}}e(\alpha \psi((du-b)/W))\bigg|\ll V(\log
t)^{-\frac{B}{16kK^2}}.
\end{align*}

Similarly,
\begin{align*}
|S_2|=&\sum_{1\leq d\leq X}\mu(d)\sum_{1\leq z\leq V/d}\sum_{1\leq
x\leq\min\{X, V/dz\}}\Lambda(x)e(\alpha
\psi((dxz-b)/W))\1_{dxz\equiv
b\pmod{W}}\\
\leq&\sum_{1\leq d\leq X}\sum_{1\leq x\leq
X}\Lambda(x)\bigg|\sum_{1\leq z\leq V/dx}e(\alpha
\psi((dxz-b)/W))\1_{dxz\equiv
b\pmod{W}}\bigg|\\
\leq&\sum_{1\leq y\leq X^2}\sum_{\substack{1\leq x\leq X\\
x\mid y}}\Lambda(x)\bigg|\sum_{1\leq z\leq V/y}e(\alpha
\psi((yz-b)/W))\1_{yz\equiv
b\pmod{W}}\bigg|\\
\leq&\log V\sum_{1\leq y\leq X^2}\bigg|\sum_{1\leq z\leq
V/y}e(\alpha
\psi((yz-b)/W))\1_{yz\equiv b\pmod{W}}\bigg|\\
\ll&V(\log V)^{1-\frac{B}{16k^2K^2}},
\end{align*}
where Lemma \ref{sumii2w} is used in the last step.

Finally, let
$$
S_3(U_1,U_2)=\sum_{U_{1}\leq u\leq U_2}\tau_u\sum_{X<x\leq
V/u}\Lambda(x)e(\alpha \psi((xu-b)/W))\1_{xu\equiv b\pmod{W}}
$$
with $X\leq U_1\leq U_2\leq 2U_1$, where
$\tau_u=\sum_{\substack{1\leq d\leq X, d\mid u}}\mu(d)$. Clearly
$S_3(U_1,U_2)\not=0$ only if $X<V/U_1$. Since $|\tau_u|\leq d(u)$,
we have
\begin{align*}
&|S_3(U_1,U_2)|\\\leq&\bigg(\sum_{U_{1}\leq u\leq
U_2}|\tau_u|^2\bigg)^\frac{1}{2}\bigg(\sum_{u=U_{1}}^{U_2}\bigg|\sum_{X<x\leq
V/u}\Lambda(x)e(\alpha \psi((xu-b)/W))\1_{xu\equiv
b\pmod{W}}\bigg|^2\bigg)^\frac{1}{2}
\\
\leq&U_2^\frac{1}{2}(\log
U_2)^\frac{3}{2}\bigg(\sum_{U_{1}\leq u\leq U_2}\sum_{\substack{X<x,y\leq V/u\\
xu\equiv b\pmod{W}\\ yu\equiv
b\pmod{W}}}\Lambda(x)\Lambda(y)e(\alpha
(\psi((xu-b)/W)-\psi((yu-b)/W)))\bigg)^\frac{1}{2}.
\end{align*}
Now for $1\leq c,c'\leq W$ with $cc'\equiv b\pmod{W}$,
\begin{align*}
&\sum_{\substack{U_1\leq u\leq U_2\\ u\equiv c\pmod{W}}}\sum_{\substack{X<x,y\leq V/u\\ x\equiv c'\pmod{W}\\ y\equiv c'\pmod{W}}}\Lambda(x)\Lambda(y)e(\alpha (\psi((xu-b)/W)-\psi((yu-b)/W)))\\
=&\sum_{\substack{X<x,y\leq V/U_1\\ x\equiv c'\pmod{W}\\ y\equiv c'\pmod{W}}}\Lambda(x)\Lambda(y)\sum_{\substack{U_{1}\leq u\leq \min\{U_2,V/x, V/y\}\\ u\equiv c\pmod{W}}}e(\alpha (\psi((xu-b)/W)-\psi((yu-b)/W)))\\
=&2\sum_{\substack{X<x<y\leq V/U_1\\ x\equiv c'\pmod{W}\\ y\equiv c'\pmod{W}}}\Lambda(x)\Lambda(y)\sum_{\substack{U_{1}\leq u\leq \min\{U_2,V/y\}\\ u\equiv c\pmod{W}}}e(\alpha (\psi((xu-b)/W)-\psi((yu-b)/W)))\\
&+O((V/U_1-X)/W).
\end{align*}
And
\begin{align*}
&\sum_{\substack{X<x<y\leq V/U_1\\ x\equiv c'\pmod{W}\\ y\equiv c'\pmod{W}}}\Lambda(x)\Lambda(y)\sum_{\substack{U_{1}\leq u\leq \min\{U_2,V/y\}\\ u\equiv c\pmod{W}}}e(\alpha (\psi((xu-b)/W)-\psi((yu-b)/W)))\\
=&\sum_{\substack{V/U_2<y\leq V/U_1\\ y\equiv c'\pmod{W}}}\Lambda(y)\sum_{\substack{X<x<y\\ x\equiv c'\pmod{W}}}\Lambda(x)\sum_{\substack{U_{1}\leq u\leq V/y\\ u\equiv c\pmod{W}}}e(\alpha (\psi((xu-b)/W)-\psi((yu-b)/W)))\\
&+\sum_{\substack{X<y\leq V/U_2\\ y\equiv
c'\pmod{W}}}\Lambda(y)\sum_{\substack{X<x<y\\ x\equiv
c'\pmod{W}}}\Lambda(x)\sum_{\substack{U_{1}\leq u\leq U_2\\
u\equiv c\pmod{W}}}e(\alpha (\psi((xu-b)/W)-\psi((yu-b)/W))).
\end{align*}
If $V/U_2<y$, then by Lemma \ref{sumii1w},
\begin{align*}
&\sum_{\substack{X<x<y\\ x\equiv c'\pmod{W}}}\bigg|\sum_{\substack{U_{1}\leq u\leq V/y\\ u\equiv c\pmod{W}}}e(\alpha (\psi((xu-b)/W)-\psi((yu-b)/W)))\bigg|\\
\ll&(X+y)U_1(\log U_1)^{-\frac{B}{16kK^2}}+(X+y)(V/y)(\log
(V/y))^{-\frac{B}{16kK^2}}\\
\ll&(X+y)(U_1+V/y)(\log V)^{-\frac{B}{16kK^2}}.
\end{align*}
Also, if $y\leq V/U_2$, then by Lemma \ref{sumii1w},
\begin{align*}
&\sum_{\substack{X<x<y\\ x\equiv c'\pmod{W}}}\sum_{\substack{U_{1}\leq u\leq U_2\\ u\equiv c\pmod{W}}}e(\alpha (\psi((xu-b)/W)-\psi((yu-b)/W)))\\
\ll& (X+y)U_1(\log U_1)^{-\frac{B}{16kK^2}}+(X+y)U_2(\log
U_2)^{-\frac{B}{16kK^2}}\\
\ll&(X+y)(U_1+U_2)(\log V)^{-\frac{B}{16kK^2}}.
\end{align*}
Hence
\begin{align*}
&\sum_{\substack{X<x<y\leq V/U_1\\ x\equiv c'\pmod{W}\\ y\equiv c'\pmod{W}}}\Lambda(x)\Lambda(y)\sum_{\substack{U_{1}\leq u\leq \min\{U_2,V/y\}\\ u\equiv c\pmod{W}}}e(\alpha (\psi((xu-b)/W)-\psi((yu-b)/W)))\\
\ll&(\log V)^{2-\frac{B}{16kK^2}}\bigg(\sum_{\substack{V/U_2<y\leq
V/U_1\\ y\equiv c'\pmod{W}}}(X+y)(U_1+V/y) +\sum_{\substack{X<y\leq
V/U_2\\ y\equiv c'\pmod{W}}}(X+y)(U_1+U_2)\bigg)\\
\ll&V^2U_1^{-1}(\log V)^{2-\frac{B}{16kK^2}}.
\end{align*}
It follows that
$$
S_3(U_1,U_2)\ll U_2^\frac{1}{2}(\log
U_2)^\frac{3}{2}(V^2U_1^{-1}(\log
V)^{2-\frac{B}{16kK^2}})^{\frac{1}{2}}\ll V(\log
V)^{3-\frac{B}{32kK^2}},
$$
and
$$
S_3\ll V(\log V)^{4-\frac{B}{32kK^2}}.
$$
All are done.
\end{proof}

\begin{Ack}
We thank Professor Emmanuel Lesigne for his useful comments on our paper.
\end{Ack}


\begin{thebibliography}{99}

\bibitem{BalogPelikanPintzSzemeredi94} A. Balog, J. Pelik\'an, J. Pintz and E. Szemer\'edi, \textit{Difference
sets without $k$-th powers}, Acta Math. Hungarica,
\textbf{65}(1994), 165-187.


\bibitem{BergelsonLeibman96} V. Bergelson and A. Leibman, \textit{Polynomial extensions of van der
Waerden's and Szemer\'edi's theorems}, J. Amer. Math. Soc.,
\textbf{9}(1996), 725-753.

\bibitem{BergelsonLesigne} V. Bergelson and E. Lesigne, \textit{Van der Corput sets in $\Z^d$}, preprint, arXiv:0710.4861.

\bibitem{Bourgain89} J. Bourgain, \textit{On $\Lambda(p)$-subsets of squares}, Israel J. Math., \textbf{67}(1989), 291-311.

\bibitem{Bourgain93} J. Bourgain, \textit{Fourier transform restriction phenomena for certain lattice subsets and applications to nonlinear evolution equations. I.
   Schr\"odinger equations}, Geom. Funct. Anal., \textbf{3}(1993), 107-156.

\bibitem{Corput39} J. G. van der Corput, \textit{\"Uber Summen von Primzahlen und
Primzahlquadraten}, Math. Ann., \textbf{116}(1939), 1-50.

\bibitem{Davenport00} H. Davenport, \textit{Multiplicative Number Theory}, Third edition, \textit{Grad.
Texts Math.} \textbf{74}, Springer-Verlag, New York, 2000.

\bibitem{FrantzikinakisHostKra} N. Frantzikinakis, B. Host and B. Kra, \textit{Multiple recurrence and
convergence for sequences related to the prime numbers}, J. Reine
Angew. Math., to appear.

\bibitem{Furstenberg1977} H. Furstenberg, \textit{Ergodic behavior of diagonal measures and a
theorem of Szemer\'edi on arithmetical progressions}, J. d'Analyse
Math., \textbf{31}(1977), 204-256.


\bibitem{Green02} B. Green, \textit{On arithmetic structures in dense sets of integers}, Duke
Math. J., \textbf{114}(2002), 215-238.

\bibitem{Green05} B. Green, \textit{Roth's theorem in the primes}, Ann. Math.
(2), \textbf{161}(2005), 1609-1636.

\bibitem{GreenTao1} B. Green and T. Tao, \textit{The primes contain arbitrarily long arithmetic
progressions}, Ann. Math., to appear.

\bibitem{GreenTao2} B. Green and T. Tao, \textit{Linear Equations in Primes}, Ann. Math., to appear.


\bibitem{Gowers01} T. Gowers, \textit{A new proof of Szemer\'edi's theorem}, Geom.
Func. Anal., \textbf{11}(2001), 465-588.

\bibitem{KamaeMendes78} T. Kamae and M. Mend\'es France, \textit{Van der Corput¡¯s difference
theorem}, Isreal J. Math., \textbf{31}(1978), 335-342.

\bibitem{Lucier06} J. Lucier, \textit{Intersective sets given by a polynomial},
Acta Arith. \textbf{123}(2006), 57-95.

\bibitem{Lucier} J. Lucier, \textit{Difference sets and shifted primes},
preprint, arXiv:0705.3749.

\bibitem{PintzSteigerSzemeredi88} J. Pintz, W. L. Steiger and E. Szemer\'edi, \textit{On sets of natural
numbers whose difference set contains no squares}, J. London Math.
Soc., \textbf{37}(1988), 219-231.

\bibitem{Roth53} K. F. Roth,
\textit{On certain sets of integers}, J. London Math. Soc.,
\textbf{28}(1953), 104--109.

\bibitem{RuzsaSanders} I. Z. Ruzsa and T. Sanders \textit{Difference sets and the primes},
preprint, arXiv:0710.0644.

\bibitem{Sarkozy78a} A. S\'ark\"ozy, \textit{On difference sets of sequences on integers I},
Acta Math. Acad. Sci. Hungar., \textbf{31}(1978), 125-149.


\bibitem{Sarkozy78b} A. S\'ark\"ozy, \textit{On difference sets of sequences on integers III},
Acta Math. Acad. Sci. Hungar., \textbf{31}(1978), 355-386.

\bibitem{Szemeredi75} E. Szemer\'edi, \textit{On sets of integers containing no k
elements in arithmetic progression}, Acta Arith.,
\textbf{27}(1975), 299-345.


\bibitem{Srinivasan85} S. Srinivasan, \textit{On a result of S\'ark\"ozy and Furstenberg},
Nieuw Arch. Wiskd. (4), \textbf{3}(1985), 275-280.


\bibitem{Tao} T. Tao, Some highlights of arithmetic combinatorics,
Lecture notes 4: {\it Roth's theorem for APs of length 3; Gowers'
proof of Szemeredi's theorem for APs of length 4}, available at\\
\texttt{http://www.math.ucla.edu/$\tilde{\;}$tao/254a.1.03w/notes4.dvi}.


\bibitem{TaoZiegler} T. Tao and T. Ziegler, \textit{The primes contain arbitrarily long
polynomial progressions}, Acta Math., to appear.

\bibitem{Vaughan97} R. C. Vaughan, The Hardy-Littlewood Method, Second edition, Cambridge University Press, Cambridge, 1997.

\bibitem{Vinogradov76} I.M. Vinogradov, Special Variants of the Method of
Trigonometric Sums (in Russian), Moscow: Nauka, 1976.




\end{thebibliography}
\end{document}